\documentclass[1 [leqno,11pt]{amsart}
\usepackage{amssymb, amsmath,amsmath,latexsym,amssymb,amsfonts,amsbsy, amsthm}

\numberwithin{equation}{section}

\allowdisplaybreaks

\let\al=\alpha

\let\f=\frac

\let\om=\omega

\let\pa=\partial

\def\cF{{\mathcal F}}
\def\cE{{\mathcal E}}
\def\cD{{\mathcal D}}
\def\cP{{\mathcal P}}

\def\ta{\widetilde{a}}
\def\tb{\widetilde{b}}

\def\tu{\widetilde{u}}

\def\R{\mathbf R}

\def\N{\mathbf N}
\def\cE{\mathcal E}
\def\cF{\mathcal F}

\def\no{\noindent}

\def\eqdef{\buildrel\hbox{\footnotesize def}\over =}
\def\ef{\hphantom{MM}\hfill\llap{$\square$}\goodbreak}

\newcommand{\beq}{\begin{equation}}
\newcommand{\eeq}{\end{equation}}
\newcommand{\ben}{\begin{eqnarray}}
\newcommand{\een}{\end{eqnarray}}
\newcommand{\beno}{\begin{eqnarray*}}
\newcommand{\eeno}{\end{eqnarray*}}

\newtheorem{theorem}{Theorem}[section]

\newtheorem{lemma}[theorem]{Lemma}
\newtheorem{proposition}[theorem]{Proposition}

\newtheorem{remark}[theorem]{Remark}

\begin{document}

\title[Blow-up criterion for the 2-D Prandtl equation]
{Blow-up criterion for the 2-D Prandtl equation}

\author{Yue Wang}
\address{School of Mathematical Sciences, Capital Normal University, Beijing 100048, China}
\email{yuewang37@pku.edu.cn}

\author{Zhifei Zhang}
\address{School of Mathematical Sciences, Peking University, 100871, Beijing, P. R. China}
\email{zfzhang@math.pku.edu.cn}

\date{\today}

\begin{abstract}
In this paper, we consider the 2-D Prandtl equation with constant outer flow and monotonic data. We prove that if the curvature of the velocity distribution(i.e., $\pa_y^2u$) is bounded near the boundary, then the solution can not develop the singularity.
\end{abstract}

\maketitle

\section{Introduction}
In this paper, we study the Prandtl equation in $\R_+\times \R^2_+$:
\begin{equation}\label{eq:Prandtl}
  \left\{
  \begin{aligned}
    &\pa_t u+u \pa_x u +v\pa_{y}u-\pa_{y}^2u+\pa_xp=0,\\
    &\pa_xu+\pa_y v=0,\\
    &u|_{y=0}=v|_{y=0}=0\quad\mbox{and}\quad \displaystyle\lim_{y\to+\infty} u(t,x,y)=U(t,x),\\
     &u|_{t=0}= u_0,
  \end{aligned}
  \right.
\end{equation}
where $(u,v)$ denotes the tangential and normal velocity of the boundary layer flow, and  $(U(t,x), p(t,x))$
{is} the values on the boundary of the tangential velocity and pressure of the outer flow, which satisfies the Bernoulli's law
\[\pa_tU+U\pa_x U+\pa_x p=0.\]
This system introduced by Prandtl  is the foundation of the boundary layer theory  \cite{Olei}.
\smallskip

Due to the derivative loss induced by nonlinear term $v\pa_yu$, the well-posedness of the Prandtl equation is still open for general data.
For monotonic data,  Oleinik \cite{Olei} proved the local
existence and uniqueness of classical solutions to the Prandtl equation. With the additional favorable pressure gradient, Xin and Zhang
\cite{XZ} proved the global existence of weak solutions.
Sammartino and Caflisch \cite{SC1, SC2} established the local well-posedness  of the Prandtl equation for analytic data, and justified the inviscid limit of the 2-D Navier-Stokes equations with nonslip boundary condition.

Recently,  Alexandre et al. \cite{AWXY} and Masmoudi-Wong \cite{MW} independently proved the well-posedness of the Prandtl equation for monotonic data in Sobolev spaces by developing direct energy method.
Chen, Wang and Zhang \cite{CWZ}  gave a simple proof based on the paralinearized technique. On the other hand, G\'{e}rard-Varet and Dormy \cite{GD} proved the ill-posedness in Sobloev spaces for the linearized Prandtl equation around a class of non-monotonic shear flows. Instead of
Sobolev spaces, G\'{e}rard-Varet and Masmoudi \cite{GM}  proved the well-posedness of the Prandtl equation (\ref{eq:Prandtl})  for a class of non-monotonic data in Gevrey class $\f 74$.  Later on, Chen-Wang-Zhang \cite{CWZ-A} and Li-Yang \cite{LY} independently proved the well-posedness in Gevrey class 2, which should be optimal by the ill-posedness result in \cite{GD}.

Up to now, we should have good understanding for the well-posedness  theory of the 2-D Prandtl equation. Then an important question is whether local (in time) smooth solution can develop a singularity(or separation) in a finite time. For analytic data without monotonicity, E and Engquist \cite{EE} proved finite time blow-up of analytic solution for the Prandtl equation. In fact, even for small analytic data, Zhang and the second author \cite{ZZ} can only establish the long time well-posedness. See also \cite{KVW, IV} for the blow-up and almost global existence result.

However, it seems more physical to consider the solution initiated from a monotonic data. In experiments and numerics \cite{Smith, CC, CSW, VS}, it was observed that the adverse pressure gradient may lead to the phenomena of the boundary layer separation for monotonic flows. To our knowledge, there seems no rigorous mathematical proof. Recently, Wang and Zhu studied the separation problem in \cite{WZ}, where for given $T>0$, outer flow and adverse pressure gradient, they proved that the solution must have a separation point $(t^*,x^*,0), t^*<T$ for a class of monotonic data, if the associated smooth solution exists in $[0,T)$. Let us mention recent important progress on the separation for the 2-D steady Prandtl equation made by Dalibard and Masmoudi  \cite{DM}.

In \cite{XZ},  the authors proved that the separation does not occur even for weak solution of the Prandtl equation under the favorable pressure gradient. Furthermore, they claimed that weak solution constructed in \cite{XZ} should be smooth. In this paper, we are concerned with whether smooth solution of the Prandtl equation can not develop a singularity as long as the separation does not occur.\smallskip

To study this question, we consider a simple case: the outer flow $U(t,x)=1, p=0$ so that the pressure gradient is favorable. We will consider a class of monotonic data $u_0(x,y)$, which holds that there exist $c, C>0$ so that
\ben\label{ass:data-H1}
\pa_yu_0(x,y)\ge ce^{-y}\quad \text{for}\quad y\in \R_+,
\een
and for $k=0,1,2$,
\ben\label{ass:data-H2}
|\pa_y^k(u_0(x,y)-u_0^s(y))|+|\pa_x\pa_yu_0(x,y)|\le Ce^{-y}\quad \text{for}\                                               (x,y)\in \R^2_+,
\een
where $u_0^s(y)\in C^2(\R_+)$ satisfies
\ben\label{ass:chi}
u_0^s(0)=0,\quad |\pa^k(u_0^s(y)-1)|\le Ce^{-y}\quad  \text{for}\  y\in \R_+  \text{ and} \ k=0,1,2.
\een
Let $u^s(t,y)$ be the solution of the heat equation
\begin{equation}\label{eq:shear}
  \left\{
  \begin{aligned}
   &\pa_t u^s -\pa_{y}^2u^s=0,\\
    & u^s|_{y=0}=0\quad\mbox{and}\quad \displaystyle\lim_{y\to\infty} u^s(t,y)=1,\\
     &u^s|_{t=0}=u_0^s(y).
  \end{aligned}
  \right.
\end{equation}
Then $(u^s(t,y),0)$ is a shear flow solution of the Prandtl equation \eqref{eq:Prandtl}. Set $\tu(t,x,y)=u(t,x,y)-u^s(t,y)$. Then $\tu$ satisfies the following equations
\begin{equation}\label{eq:Prandtl-3}
  \left\{
  \begin{aligned}
    &\pa_t \tu+u\pa_x\tu+v\pa_{y}u-\pa_{y}^2\tu=0,\\
    &\pa_x\tu+\pa_y v=0,\\
   &\lim_{y\to+\infty}\tu(t,x,y)=0,\quad v(t,x,y)|_{y=0}=0,\\
   &\tu|_{t=0}= \tu_0(x,y)=u_0(x,y)-u_0^s(y).
  \end{aligned}
  \right.
\end{equation}

In the first part of this paper, we prove the local well-posedness  of the Prandtl equation in Sobolev space and give a blow-up criterion of the solution.

\begin{theorem}\label{thm:main1}
Assume that the initial data satisfies \eqref{ass:data-H1}-\eqref{ass:data-H2} and  $\tu_0\in H^{3,1}_\mu\cap H^{1,2}_\mu$ with $\tu_0(x,0)=0$. There exists $T>0$ and a unique solution of  the Prandtl equation  (\ref{eq:Prandtl}) in  $[0,T]$, which satisfies
\beno
&&c_1e^{-y}\le \pa_yu(t,x,y)\le C_1e^{-y}\quad \textrm{for }\,\,(t,x,y)\in [0,T]\times\R^2_+,\\
&&|\pa_y^2u(t,x,y)|+|\pa_x\pa_yu(t,x,y)|\le C_1e^{-y}\quad \textrm{for }\,\,(t,x,y)\in [0,T]\times\R^2_+,\\
&&\tu\in L^\infty\big(0,T;H^{3,1}_\mu\cap H^{1,2}_\mu\big)\cap
L^2\big(0,T; H^{1,3}_\mu\big),
\eeno
for some $c_1, C_1>0$. Let $T^*$ be the maximal existence time of the solution. If the solution satisfies the following conditions
\ben
&c e^{-y}\le \pa_yu(t,x,y)\le Ce^{-y}\quad \textrm{for }\,\,(t,x,y)\in [0,T^*)\times\R^2_+,\label{ass:mono}\\
&\sup_{t\in [0,T^*)}A(t)\le C,\label{ass:A}
\een
for some $c, C>0$,  then it can be extended after $t=T^*$. Here
\begin{align*}
A(t)=\sup\limits_{(s,x,y)\in [0,t]\times\R^2_+}e^y\Big(\sum_{k=0}^2|\pa_{y}^k{(u-1)}|+(1+y)^{-1}\big(|\pa_x u|+|\pa_x\pa_y u|\big)\Big),\end{align*}
and $H^{k,\ell}_\mu$ is the weighted anisotropic Sobolev space defined
in next section.
\end{theorem}

In the following theorem, we study whether the conditions \eqref{ass:mono} and \eqref{ass:A} hold for the local smooth solution constructed in Theorem \ref{thm:main1}.

\begin{theorem}\label{thm:main2}
Let $u$ be a solution in $[0,T^*)$ constructed in Theorem \ref{thm:main1}.
Then there exists $c_0, C_0$ so that
\beno
c_0e^{-y}\le \pa_yu(t,x,y)\le C_0e^{-y}\quad \textrm{for }\,\,(t,x,y)\in [0,T^*)\times\R^2_+.
\eeno
Moreover, if
\ben\label{ass:cur}
|\pa_y^2u(t,x,y)|\le C \quad \textrm{for }\,\,(t,x)\in [0,T^*)\times\R, \ y\in [0,\delta]\label{ass:sep}
\een
for some $C>0, \delta>0$, then it holds that
\beno
A(t)\le C_*\quad \text{for any}\,\,t\in \big[0,T^*\big),
\eeno
where $C_*$ is a constant depending only on $C_0, c_0, T^*, \delta$.
\end{theorem}

\begin{remark}
The condition \eqref{ass:cur} means that the curvature of the velocity distribution is bounded near the boundary. This result shows that it seems possible that the curvature of the velocity distribution blows up near the boundary $y=0$ even if the separation does not occur.\smallskip

To control the part of $y\ge \delta$ in $A(t)$, the condition \eqref{ass:cur}
is unnecessary.
 \end{remark}

The proof of Theorem \ref{thm:main1}  used  the paralinearization method
so that we can control the higher order Sobolev norm of the solution in terms of the lower order derivatives of the solution, i.e., $A(t)$. To avoid the derivative loss in the energy estimate process, the key idea is  to introduce a good unknown $w=\pa_y\Big(T_{\f 1 {\pa_yu}}\tu\Big)$ inspired by \cite{AWXY, MW}, which satisfies a good equation like
\beno
\pa_t w+T_u \pa_x w -\pa_{y}^2 w=F.
\eeno
Here $T_u$ is the Bony's paraproduct.\smallskip

The proof of Theorem \ref{thm:main2} used the different framework. As in \cite{XZ}, we will use the Crocco transformation which transforms the Prandtl equation into a quasilinear degenerate parabolic system. Using the interior regularity results for the kinetic type equation from \cite{WZL, GI, IM}, the interior gradient estimate is relatively easy. The boundary gradient estimates are highly non-trivial. For this, we will use the maximum principle by constructing good auxiliary functions.  The estimate near $y=0$ is more delicate, since the new unknown $w=\pa_y u$ satisfies the Neumann boundary condition at $y=0$. This is main reason why we need to impose the extra condition
\eqref{ass:cur} near $y=0$.

\section{Functional spaces and Paraproduct}

\subsection{Weighted anisotropic Sobolev spaces}

Let $\om(y)$ be a nonnegative function in $\R_+$.
We introduce the weighted $L^p$ norm:
\beno
\|f\|_{L^p_\om}=\|\om(y)f(x,y)\|_{L^p},\quad \|f\|_{L^p_{y,\om}}= \|\om(y)f(y)\|_{L^p}.
\eeno

Let $k,\ell\in \N$, the weighted anisotropic Sobolev space $H^{k,\ell}_\om$ consists of all functions $f\in L^2_\om$ satisfying
\beno
\|f\|_{H^{k,\ell}_\om}^2=\sum_{\al\le k}\sum_{\beta\le \ell}\|\pa_x^\al\pa_y^\beta f\|_{L^2_\om}^2<+\infty.
\eeno
We denote by $H^\ell_{y,\om}$ the weighted Sobolev space in $\R_+$, which consists of all functions $f\in L^2_{y,\om}$ satisfying
\beno
\|f\|_{H^{\ell}_{y,\om}}^2\eqdef \sum_{\beta\le \ell}\|\pa_y^\beta f\|_{L^2_{y,\om}}^2<+\infty.
\eeno
When $\om=1$, we denote $H^{k,\ell}_\om$ by $H^{k,\ell}$, and $H^\ell_{y,\om}$ by $H^\ell_y$ for the simplicity. \smallskip

The following interpolation inequality will be used: if $\pa_y^\ell u|_{y=0}$ or $\pa_y^{\ell-1}u|_{y=0}=0$, then we have
\ben\label{eq:interpolation}
\|u\|_{H^{k,\ell}}\le \|u\|_{H^{k-1,\ell+1}}^\f12\|u\|_{H^{k+1,\ell-1}}^\f12.
\een

In this paper, we will use the following weights
\beno
\mu(y)=e^{\f y2},\quad \nu(y)=e^{-\f y2},\quad \om(y)=e^{\f 23y}.
\eeno

\subsection{Paraproduct}
We first introduce the paraproduct decomposition in $\R$.
We define
\beno
\begin{split}
&\Delta_jf=\cF^{-1}(\varphi(2^{-j}\xi)\widehat{f})\quad
S_jf=\cF^{-1}(\chi(2^{-j}\xi)\widehat{f}\quad \textrm{for}\,\, j\ge 0,\\
&\Delta_{-1}f=S_0f, \quad S_jf=S_0f\quad \textrm{for}\,\, j<0,
\end{split}
\eeno
where $\cF f$ and $\widehat{f}$ always denote the partial  Fourier transform of $f$ with respect to $x$ variable, and $\chi(\tau),$ ~$\varphi(\tau)$ are smooth functions such that
 \beno
&&\textrm{supp} \varphi \subset \Bigl\{\tau \in \R:  \frac34 \leq
|\tau| \leq \frac83 \Bigr\},\quad \textrm{supp}\chi \subset \Bigl\{\tau \in \R: \ |\tau|  \leq
\frac43 \Bigr\},
 \eeno
 and for any $\tau\in \R$,
\beno
\chi(\tau)+ \sum_{j\geq 0}\varphi(2^{-j}\tau)=1.
\eeno
The Bony's paraproduct $T_fg$ is defined by
\beno
T_fg=\sum_{j\ge-1}S_{j-1}f\Delta_jg.
\eeno
Then we have the following Bony's decomposition
\ben\label{Bony}
fg=T_{f}g+R_gf,
\een
where the remainder term $R_gf$ is defined by
\beno
R_gf=\sum_{j\ge 0}\Delta_jfS_1g+\sum_{j\ge 1, j'\ge j-1}\Delta_{j'}f\Delta_{j}g.
\eeno

Next we recall the classical paraproduct estimate and paraproduct calculus in Sobolev space \cite{BCD}.  We denote by $W^{s,p}$ the usual Sobolev spaces.

\begin{lemma}\label{lem:para-p}
Let $s\in \R$. It holds that
\beno
\|T_{f}g\|_{H^s}\le C\|f\|_{L^\infty}\|g\|_{H^s}.
\eeno
If $s>0$, then we have
\beno
\|R_fg\|_{H^s}\le C\|f\|_{L^\infty}\|g\|_{H^s}.
\eeno
\end{lemma}

\begin{lemma}\label{lem:para-cal}
Let $s\in \R$ and $\sigma\in (0,1]$. It holds that
\beno
\|(T_a T_{b}-T_{ab})f\|_{H^s}
\le C\big(\|a\|_{W^{\sigma,\infty}}\|b\|_{L^\infty}+\|a\|_{L^\infty}\|b\|_{W^{\sigma,\infty}}\big)\|f\|_{H^{s-\sigma}}.
\eeno
Especially, we have
\beno
&&\|[T_a,T_b]f\|_{H^s}\le  C\big(\|a\|_{W^{\sigma,\infty}}\|b\|_{L^\infty}+\|a\|_{L^\infty}\|b\|_{W^{\sigma,\infty}}\big)\|f\|_{H^{s-\sigma}}.
\eeno
\end{lemma}

Let us conclude this section by the following basic estimates
for the heat equation. See \cite{CWZ} for example.

\begin{lemma}\label{lem:shear}
Let $u^s(t,y)$ be the solution of \eqref{eq:shear}. Then it holds that for $k=0,1,2$ and any $T<+\infty$,
\beno
|\pa^k(u^s(t,y)-1)|\le C_se^{-y}\quad\text{for}\ (t,y)\in [0,T]\times \R_+,
\eeno
where the constant $C_s$ depends on $T$ and $u_0^s$.
\end{lemma}

\section{Uniform  tame energy estimates}

For this part, the proof is similar to \cite{CWZ}. However, we need to refine
the energy estimates there in order to prove the blow-up criterion.

Let us introduce the energy functional
\beno
E(t)=\|\tu(t)\|_{H^{1,2}_\mu}^2+\|\tu(t)\|_{H^{3,1}_{\mu}}^2+\|\tu\|_{H^{2,0}_\om}^2,
\eeno
and
\begin{align*}
A(t)=\sup\limits_{(s,x,y)\in [0,t]\times\R^2_+}e^y\Big(\sum_{k=0}^2|\pa_{y}^k{\widetilde{u}}(s,x,y)|+(1+y)^{-1}\big(|\pa_x u(s,x,y)|+|\pa_x\pa_y u(s,x,y)|\big)\Big).\end{align*}
Here the definition of $A(t)$ is slightly different from that in \cite{CWZ} with
an extra factor $(1+y)^{-1}$ in the front of $|\pa_xu|$ and $|\pa_{xy}u|$.

In the sequel,  we assume that
\beno
\pa_y u(t,x,y)\ge c_1e^{-y}\quad \textrm{for}\quad (t,x,y)\in [0,T]\times \R^2_+.
\eeno
We denote by $C$ a constant depending only on $c_1, C_s, T$.

\subsection{Sobolev estimate in horizontal direction}

We denote
\beno
a=1/{\pa_y u},\quad b=\pa_y u.
\eeno
We first paralinearize the first equation of (\ref{eq:Prandtl-3}).  Using Bony's decomposition (\ref{Bony}), we obtain
\begin{align*}
\pa_t \tu+T_u\pa_x \tu+T_{\pa_y u}v-\pa_{y}^2\tu=f,
\end{align*}
where
\beno
f= -R_{\pa_x\tu}\tu-R_{v}\pa_y u.
\eeno
Notice that there is no derivative loss for the terms in $f$.
To eliminate the trouble term $T_{\pa_y u}v$,  it is natural to introduce a good unknown $w$ defined by
\beno
w=\pa_yT_a\tu.
\eeno
A direct calculation shows
\begin{equation}\label{eq:Prandtl-w}
  \left\{
  \begin{aligned}
    &\pa_t w+T_u \pa_x w -\pa_{y}^2 w=\pa_y F_1+F_2,\\
    & \pa_yw|_{y=0}=0\quad \mbox{and}\quad\displaystyle\lim_{y\to\infty} w=0,\\
     &w|_{t=0}= w_0(x,y),
  \end{aligned}
  \right.
\end{equation}
where $F_1$ and $F_2$ are given by
\beno
&&F_1=T_{\pa_ta}\tu-[T_a,T_u]\pa_x\tu+(T_{ab}-T_aT_b)v-[\pa_{y}^2,T_a]\tu
+T_af+T_uT_{\pa_x a}\tu,\\
&&F_2=(T_{ba}-T_bT_a)\pa_x \tu+T_{\pa_y u}T_{\pa_x a}\tu.
\eeno


\begin{lemma}\label{lem:F-S}
It holds that
\beno
\|F_1\|_{H^{3,0}_{\nu}}+\|F_2\|_{H^{3,0}_{\nu}}
\le C\big(1+A(t)^2\big)E(t)^\f12.
\eeno
\end{lemma}

\begin{proof}
Let $\ta=a-(\pa_y u^s)^{-1}$ and $\tb=\pa_y u-\pa_yu^s$.
Then we have
\beno
[T_a,T_u]=[T_{\ta},T_{\tu}],\quad T_{ab}-T_aT_b=T_{\ta\tb}-T_{\ta}T_{\tb}.
\eeno
It follows from Lemma \ref{lem:para-cal} and Lemma \ref{lem:shear}  that
\begin{align*}
\|[T_a,T_u]\pa_x\tu\|_{H^{3,0}_{\nu}}&=\|e^{-\f{y}{2}}[T_a,T_u]\pa_x\tu\|_{H^{3,0}}=\big\|[T_{e^{-y}\tilde{a}},T_{\tilde{u}}]e^\f{y}{2}\pa_x\tu\big\|_{H^{3,0}}\\
&\le C\|(1+y)^{-1}e^{-y}\ta\|_{L^\infty_y(W^{1,\infty}_x)}\|(1+y)\tu\|_{L^\infty_y(W^{1,\infty}_x)}\|\tu\|_{H^{3,0}_\mu}\\
&\le C\big(1+A(t)^2\big)\|\tu\|_{H^{3,0}_\mu}.
\end{align*}
Similarly, we have
\begin{align*}
\|(T_{ab}-T_aT_b)v\|_{H^{3,0}_{\nu}}
\le& C\|(1+y)^{-1}e^{-y}\ta\|_{L^\infty_y(W^{1,\infty}_x)}\|(1+y)e^{\f{y}{2}}\tb\|_{L^2_y(W^{1,\infty}_x)}\|v\|_{H^{2,0}}\\
\le& C\big(1+A(t)^2\big)\|\tu\|_{H^{3,0}_\mu}.
\end{align*}
Here we used $v=-\int_0^y\pa_x \tu dy'$. By Lemma \ref{lem:para-p}, we have
\begin{align*}
\|T_af\|_{H^{3,0}_{\nu}}\le& C\|e^{- y}a\|_{L^\infty}\|f\|_{H^{3,0}_\mu}\\
\le& C\big(\|\pa_x\tu\|_{L^\infty}\|\tu\|_{H^{3,0}_\mu}+\|v\|_{L^\infty}\|\tu\|_{H^{3,1}_\mu}+\|\tu\|_{H^{2,0}_\mu}\big)\\
\le& C\big(1+A(t)\big)\|\tu\|_{H^{3,0}_\mu}.
\end{align*}

We denote
\begin{align*}
F_{11}=&T_{\pa_ta}\tu-[\pa_{y}^2,T_a]\tu+T_uT_{\pa_x a}\tu\\
=&T_{\pa_t a-\pa_y^2a}\tu-2T_{\pa_y a}\pa_y \tu-T_uT_{\f {\pa_x\pa_y\tu} {(\pa_yu)^2}}.
\end{align*}
Using the first equation of (\ref{eq:Prandtl-3}), we find that
\begin{align*}
\pa_t a-\pa_y^2a=&-\frac {(\pa_t-\pa_y^2)\pa_yu} {(\pa_y u)^2}-\frac {2(\pa_y^2 u)^2} {(\pa_y u)^3}\\
=&\frac {u\pa_{x}\pa_y\tu} {(\pa_y u)^2}+\f{v\pa_y^2u} {(\pa_y u)^2}-\frac {2(\pa_y^2 u)^2} {(\pa_y u)^3}.
\end{align*}
Then we have
\beno
F_{11}=T_{\frac {\tu\pa_{x}\pa_y\tu} {(\pa_y u)^2}}\tu-T_{\tu}T_{\frac {\pa_{x}\pa_y\tu} {(\pa_y u)^2}}\tu+T_{\f{v\pa_y^2u} {(\pa_y u)^2}}\tu-T_{\frac {2(\pa_y^2 u)^2} {(\pa_y u)^3}}\tu-2T_{\pa_y a}\pa_y \tu,
\eeno
from which and Lemma \ref{lem:para-p}, we infer that
\beno
\|F_{11}\|_{H^{3,0}_{\nu}}\le C\big(1+A(t)^2\big)E(t)^\f12.
\eeno

Summing up, we conclude that
\beno
\|F_1\|_{H^{3,0}_{\nu}}\le C\big(1+A(t)^2\big)E(t)^\f12.
\eeno
Similarly, we have
\beno
\|F_2\|_{H^{3,0}_{\nu}}\le C\big(1+A(t)^2\big)\|\tu\|_{H^{3,0}_\mu}.
\eeno
This proves our result.
\end{proof}

\begin{proposition}\label{prop:energy-w-s}
Let $w$ be a smooth solution of (\ref{eq:Prandtl-w}) in $[0,T]$. Then it holds that for any $t\in [0,T]$,
\begin{align*}
\frac d{dt}\|w\|_{H^{3,0}_{\nu}}^2+\|\pa_yw\|_{H^{3,0}_{\nu}}^2\le& C\big(1+A(t)^4\big)E(t)+C\|w\|_{H^{3,0}_\nu}^2.
\end{align*}
\end{proposition}

\begin{proof}
Making $H^{3,0}_{\nu}$ energy estimate to (\ref{eq:Prandtl-w}), we obtain
\begin{align*}
&\frac12\frac d{dt}\|w\|_{H^{3,0}_{\nu}}^2-(\pa_{y}^2 w,w)_{H^{3,0}_{\nu}}+\big(T_u \pa_x w,w\big)_{H^{3,0}_{\nu}}=\big(\pa_yF_1, w\big)_{H^{3,0}_{\nu}}+\big(F_2, w\big)_{H^{3,0}_{\nu}}.
\end{align*}

Thanks to $\pa_y w|_{y=0}=0$, we get by integration by parts that
\begin{align*}
-\big(\pa_{y}^2 w, w\big)_{H^{3,0}_{\nu}}=\|\pa_yw\|_{H^{3,0}_{\nu}}^2-(\pa_y w,w)_{H^{3,0}_{\nu}}\geq \f12\|\pa_yw\|_{H^{3,0}_{\nu}}^2-C\|w\|_{H^{3,0}_{\nu}}^2.
\end{align*}
Notice that
\begin{align*}
\big(T_u \pa_x w, w\big)_{H^{3,0}_{\nu}}=D+\sum_{k=1}^{3}\big([\pa_x^k, T_{\tu}] \pa_x w, \pa_x^k w\big)_{L^2_{\nu}},
\end{align*}
where
\begin{align*}
D=&-\f12\sum_{k=0}^{3}\big( T_{\pa_x\tu} \pa_x^k w, \pa_x^k w\big)_{L^2_{\nu}}+\f 12\sum_{k=0}^{3}\big( (T_u-T_u^*) \pa_x^k\pa_x w, \pa_x^k w\big)_{L^2_{\nu}}.
\end{align*}
Then by Lemma \ref{lem:para-p} and  Lemma \ref{lem:para-cal}, we get
\begin{align*}
\big(T_u \pa_x w, w\big)_{H^{3,0}_\nu}
\le CA(t)\|w\|_{H^{3,0}_{\nu}}^2.
\end{align*}

Thanks to $u=v=0$ and $\pa_y^2\tu=0$ on $y=0$,  we have
$F_1|_{y=0}=0.$ Then by integration by parts and Lemma \ref{lem:F-S}, we get
\begin{align*}
\big(\pa_yF_1,w\big)_{H^{3,0}_{\nu}}+\big(F_2,w\big)_{H^{3,0}_{\nu}}
\le& C\|F_1\|_{H^{3,0}_{\nu}}\|\pa_yw\|_{H^{3,0}_{\nu}}+\|F_2\|_{H^{3,0}_{\nu}}\|w\|_{H^{3,0}_{\nu}}\\
\le& C\big(1+A(t)^4\big)E(t)+C\|w\|_{H^{3,0}_\nu}^2+\f14\|\pa_yw\|_{H^{3,0}_\nu}^2.
\end{align*}

Summing up, we conclude the desired energy estimate.
\end{proof}

Using the equation \eqref{eq:Prandtl-3}, it is easy to prove the following
energy estimate of $\tu$ in lower order Sobolev space  with more exponential decay in $y$.

\begin{proposition}\label{prop:u-lower}
Let $\tu$ be a smooth solution of \eqref{eq:Prandtl-3} in $[0,T]$.
Then it holds that for any $t\in [0,T]$,
\begin{align*}
\frac d{dt}\|\tu\|_{H^{2,0}_{\om}}^2+\|\pa_y\tu\|_{H^{2,0}_{\om}}^2\le& C\big(1+A(t)^2\big)E(t).
\end{align*}

\end{proposition}

\subsection{Sobolev estimate in vertical direction}

To close the energy estimates, we need to derive the high order derivative estimates in the vertical variable $y$. For this part, we don't need to use the monotonicity of the solution.

\begin{proposition}\label{prop:u-energy-s}
Let $\tu$ be a smooth solution of (\ref{eq:Prandtl-3}) in $[0,T]$. It holds that for any $t\in [0,T]$,
\begin{align*}
&\f d {dt}\Big(\|\tu\|_{H_\mu^{1,0}}^2+\|\pa_y\tu\|_{H_\mu^{1,0}}^2+\|\tu_t\|_{H_\mu^{1,0}}^2\Big)+\Big(\|\pa_y\tu\|_{H_\mu^{1,0}}^2+\|\tu_t\|_{H_\mu^{1,0}}^2+\|\pa_y\tu_t\|_{H_\mu^{1,0}}^2\Big)\\
&\le C\big(1+A(t)^4\big)\big(E(t)+\|\tu_t\|_{H^{1,0}_\mu}^2\big)+C\big(1+A(t)\big)\|\tu\|_{H_\mu^{2,2}}\|\tu_t\|_{H^{1,0}_\mu}.
\end{align*}
\end{proposition}

\begin{proof}

We follow the proof of Proposition 6.1 in \cite{CWZ} step by step again.
First of all, we have
\begin{align}\label{eq:u-L2}
&\frac d{dt}\|\tu\|_{H_\mu^{1,0}}^2+\|\pa_y\tu\|_{H_\mu^{1,0}}^2
\le C\big(1+A(t)\big)E(t),\\
\label{eq:u-H1}
&\frac d{dt}\|\pa_y\tu\|_{H_\mu^{1,0}}^2+\|\pa_t\tu\|_{H_\mu^{1,0}}^2
\le C\big(1+A(t)^2\big)E(t).
\end{align}

Next we present $H^{1,2}_\mu$ estimate in detail.
Taking the time derivative to (\ref{eq:Prandtl-3}), we obtain
\beno
\pa_t \tu_t-\pa_{y}^2\tu_t=-u_t\pa_x \tu-v_t\pa_{y}u-u \pa_x \tu_t-v\pa_{y}u_t.
\eeno
Making $H^{1,0}_\mu$ inner product with $\tu_t$, we get
\begin{align*}
\f12\frac d{dt}\|\tu_t\|_{H_\mu^{1,0}}^2-\big(\pa_y^2\tu_t,\tu_t\big)_{H_\mu^{1,0}}
=&-\big(u_t\pa_x\tu,\tu_t\big)_{H_\mu^{1,0}}-\big(v_t\pa_yu,\tu_t\big)_{H_\mu^{1,0}}\\
&-\big(u\pa_x\tu_t,\tu_t\big)_{H_\mu^{1,0}}-\big(v\pa_yu_t,\tu_t\big)_{H_\mu^{1,0}}.
\end{align*}
We get by integration by parts that
\begin{align*}
-\big(\pa_y^2\tu_t,\tu_t\big)_{H_\mu^{1,0}}\ge\frac12\|\pa_y\tu_t\|_{H_\mu^{1,0}}^2-C\|\tu_t\|_{H_\mu^{1,0}}^2.
\end{align*}
Thanks to $\|u_t\|_{L^\infty}\le C\big(1+A(t)^2\big)$(using \eqref{eq:Prandtl}),  it is easy to see that
\begin{align*}
\big(u_t\pa_x\tu,\tu_t\big)_{H_\mu^{1,0}}\le& \|u_t\pa_x\tu\|_{H_\mu^{1,0}}\|\tu_t\|_{H^{1,0}_\mu}\\
\leq&
C\big(A(t)\|\tu_t\|_{H_\mu^{1,0}}+\|u_t\|_{L^\infty}\|\tu\|_{H_\mu^{2,0}}\big)\|\tu_t\|_{H^{1,0}_\mu}\\
\le& C\big(1+A(t)^2\big)\big(E(t)+\|\tu_t\|_{H^{1,0}_\mu}^2\big),
\end{align*}
and
\begin{align*}
\big(v_t\pa_yu,\tu_t\big)_{H_\mu^{1,0}}\le& \|v_t\pa_yu\|_{H^{1,0}_\mu}\|\tu_t\|_{H^{1,0}_\mu}\\
\le& C\big(\|v_t\pa_yu^s\|_{H^{1,0}_\mu}+\|v_t\pa_y\tu\|_{H^{1,0}_\mu}\big)\|\tu_t\|_{H^{1,0}_\mu}\\
\leq&
C\big(\|v_t\|_{H^1_xL_y^\infty}\|\pa_yu^s\|_{L^2_{y,\mu}} +\|v_t\|_{H^1_xL_y^\infty}\|\pa_y\tu\|_{L^2_{y,\mu}(W^{1,\infty}_x)}\big)\|\tu_t\|_{H^{1,0}_\mu}\\
\leq&
C\big(1+A(t)\big)\|\tu_t\|_{H_\mu^{2,0}}\|\tu_t\|_{H^{1,0}_\mu}.
\end{align*}
Thanks to $\pa_xu+\pa_y v=0$, we get by integration by parts that
\begin{align*}
&\big(u\pa_x\tu_t,\tu_t\big)_{H_\mu^{1,0}}+\big(v\pa_yu_t,\tu_t\big)_{H_\mu^{1,0}}\\
&=\big(v\pa_yu_t^s,\tu_t\big)_{H_\mu^{1,0}}
+\big(\pa_x\tu\pa_x\tu_t,\pa_x\tu_t\big)_{L^2_\mu}+\big(\pa_xv\pa_yu_t,\pa_x\tu_t\big)_{L^2_\mu}\\
&\quad-\f 12\big(v\tu_t,\tu_t\big)_{L^2_\mu}-\f 12\big(v\pa_x\tu_t,\pa_x\tu_t\big)_{L^2_\mu}\\
&\le C\big((1+A(t))\|\tu_t\|_{H^{1,0}_\mu}^2+C\big(1+A(t)^4\big)\|\tu\|^2_{H^{2,0}_\mu}+\f14\|\pa_y\tu_t\|_{H^{1,0}_\mu},
\end{align*}
here we used
\beno
\big(\pa_xv\pa_yu_t,\pa_x\tu_t\big)_{L^2_\mu}=\big(\pa_x^2\tu u_t,\pa_x\tu_t\big)_{L^2_\mu}-\big(\pa_xv u_t,\pa_y\pa_x\tu_t\big)_{L^2_\mu}
-\big(\pa_xv u_t,\pa_x\tu_t\big)_{L^2_\mu}.
\eeno
Thus, we deduce that
\begin{align}\label{eq:u-H2}
\frac d{dt}\|\tu_t\|_{H_\mu^{1,0}}^2+\|\pa_y\tu_t\|_{H_\mu^{1,0}}^2
\le& C\big(1+A(t)^4\big)\big(E(t)+\|\tu_t\|_{H^{1,0}_\mu}^2\big)\\
&+C\big(1+A(t)\big)\|\tu_t\|_{H_\mu^{2,0}}\|\tu_t\|_{H^{1,0}_\mu}.\nonumber
\end{align}

Summing up (\ref{eq:u-L2})--(\ref{eq:u-H2}), we conclude that
\begin{align}\label{eq:u-total}
&\f d {dt}\Big(\|\tu\|_{H_\mu^{1,0}}^2+\|\pa_y\tu\|_{H_\mu^{1,0}}^2+\|\tu_t\|_{H_\mu^{1,0}}^2\Big)+\Big(\|\pa_y\tu\|_{H_\mu^{1,0}}^2+\|\tu_t\|_{H_\mu^{1,0}}^2+\|\pa_y\tu_t\|_{H_\mu^{1,0}}^2\Big)\nonumber\\
&\le C\big(1+A(t)^4\big)\big(E(t)+\|\tu_t\|_{H^{1,0}_\mu}^2\big)+C\big(1+A(t)\big)\|\tu_t\|_{H_\mu^{2,0}}\|\tu_t\|_{H^{1,0}_\mu}.
\end{align}

It remains to estimate $\|\tu_t\|_{H^{2,0}_\mu}$. Using the first equation of (\ref{eq:Prandtl-3}),
we get
\begin{align}
\|\tu_t\|_{H^{2,0}_\mu}
\le& \|\tu\|_{H^{2,2}_\mu}+\|u\pa_x\tu\|_{H^{2,0}_\mu}+\|v\pa_y u\|_{H^{2,0}_\mu}\nonumber\\
\le& \|\tu\|_{H^{2,2}_\mu}+C\big(1+A(t)\big)\big(\|\tu\|_{H^{3,0}_\mu}+\|\tu\|_{H^{2,1}_\mu}\big),\nonumber
\end{align}
from which and  (\ref{eq:u-total}), we conclude our result.
\end{proof}

\subsection{Relation between good unknown $w$ and $\tu$}

To recover the estimates of $\tu$ from those of $w$, we need the following lemma.

\begin{lemma}\label{lem:relation}
It holds that
\beno
&&\|\tu\|_{H^{3,0}_\mu}\le C\big(1+A(t)^2\big)\big(\|w\|_{H^{3,0}_\nu}+\|\tu\|_{H^{2,0}_\om}\big),\\
&&\|\tu\|_{H^{3,1}_\mu}\le C\big(1+A(t)^4\big)\big(\|w\|_{H^{3,0}_\nu}+\|\tu\|_{H^{2,0}_\om}+\|\tu\|_{H^{2,0}_\om}^\f12\|\tu\|_{H^{2,2}_\mu}^\f12\big),\\
&&\|\tu\|_{H^{1,2}_\mu}\le \|\tu_t\|_{H^{1,0}_\mu}+C\big(1+A(t)\big)\|\tu\|_{H^{2,0}_\mu},\\
&&\|\tu\|_{H^{1,3}_\mu}\le \|\pa_y\tu_t\|_{H^{1,0}_\mu}+\|\tu_t\|_{H^{1,0}_\mu}+C\big(1+A(t)\big)\|\tu\|_{H^{2,1}_\mu}.
\eeno
\end{lemma}

\begin{proof}
To control $\|\tu\|_{H^{3,0}_\mu}$, we use the formula
\beno
\tu=T_bT_a\tu+\big(T_{ba}-T_{b}T_a\big)\tu=T_b\Big(\int_0^ywdy'\Big)+\big(T_{\tb\ta}-T_{\tb}T_{\ta}\big)\tu,
\eeno
which along with Lemma \ref{lem:para-p} and Lemma \ref{lem:para-cal}
gives
\begin{align*}
\|\tu\|_{H^{3,0}_\mu}\le C\big(1+A(t)\big)\|w\|_{H^{3,0}_{\nu}}+C\big(1+A(t)^2\big)\|\tu\|_{H^{2,0}_\om}.
\end{align*}
Here we used
\beno
\Big\|e^{-\f y2 }\int_{0}^{y}wdy'\Big\|_{H^{3,0}}= \Big\|\int_{0}^{y}e^{-\f{y-y'}{2} }e^{-\f {y'}{2} }wdy'\Big\|_{H^{3,0}}\leq C\|w\|_{H^{3,0}_{\nu}}.
\eeno

Using the relation
\beno
\pa_y\tu=T_bw+T_bT_{\pa_y^2u/\pa_y u}\tu+\big(T_{ab}-T_bT_a\big)\pa_y\tu,
\eeno
we infer from Lemma \ref{lem:para-p} and Lemma \ref{lem:para-cal} that
\beno
\|\pa_y\tu\|_{H^{3,0}_\mu}\le C\big(1+A(t)^2\big)\big(\|w\|_{H^{3,0}_\nu}+\|\tu\|_{H^{3,0}_\mu}\big)+C\big(1+A(t)^2\big)\|(1+y)\pa_y\tu\|_{H^{2,0}_\mu}.
\eeno
By the interpolation, we have
\beno
\|(1+y)\pa_y\tu\|_{H^{2,0}_\mu}\le C\|\tu\|_{H^{2,0}_\om}^\f12\|\tu\|_{H^{2,2}_\mu}^\f12.
\eeno
Thus, we conclude that
\beno
\|\tu\|_{H^{3,1}_\mu}\le C\big(1+A(t)^4\big)\big(\|w\|_{H^{3,0}_\nu}+\|\tu\|_{H^{2,0}_\om}+\|\tu\|_{H^{2,0}_\om}^\f12\|\tu\|_{H^{2,2}_\mu}^\f12\big).
\eeno

Using the first equation of (\ref{eq:Prandtl-3}), we get
\begin{align*}
\|\tu\|_{H^{1,2}_\mu}\le& \|\tu_t\|_{H^{1,0}_\mu}+\|u\pa_x\tu\|_{H^{1,0}_\mu}+\|v\pa_y u\|_{H^{1,0}_\mu}\\
\le& \|\tu_t\|_{H^{1,0}_\mu}+C\big(1+A(t)\big)\|\tu\|_{H^{2,0}_\mu}.
\end{align*}

Using (\ref{eq:Prandtl-3}) again, we have
\begin{align*}
\|\pa_y^3\tu\|_{H^{1,0}_\mu}\le& \|\pa_y\tu_t\|_{H^{1,0}_\mu}+\|u\pa_x\tu\|_{H^{1,1}_\mu}+\|v\pa_y u\|_{H^{1,1}_\mu}\\
\le& \|\pa_y\tu_t\|_{H^{1,0}_\mu}+C\big(1+A(t)\big)\|\tu\|_{H^{2,1}_\mu},\end{align*}
which implies the last inequality.
\end{proof}

\section{Monotonicity and no separation}

In this section, we prove that the separation can not occur for the solution
constructed in Theorem \ref{thm:main1} in $[0,T^*)$, where $T^*$ is the maximal existence time.  To this end, we need to use the Crocco transformation defined by
\begin{align*}\label{Crocco}
    t=\tau,\  \xi=x,\  \eta=u(x,y,\tau),\  w(\eta,\xi,t)=\partial_yu(x,y,\tau).\end{align*}
Then the initial-boundary value problem \eqref{eq:Prandtl}  becomes
\begin{equation}\label{eq:Pran-Cro}
\left\{\begin{aligned}
&-\partial_t  w-\eta \partial_\xi w+w^2\partial_{\eta \eta} w=0\quad (\xi,\eta)\in \R\times (0,1),\\
&\partial_\eta w\mid _{\eta=0}=0,\quad w\mid _{\eta=1}=0.
\end{aligned}\right.
\end{equation}

\begin{proposition} \label{prop:mono}
Let $w$ be a solution of \eqref{eq:Pran-Cro} in $[0,T^*)$, which  satisfies
\begin{align}\label{condition}\begin{split}
    &c(1-\eta)\leq w\leq C(1-\eta)\quad on\quad\{t=0\}\cup\{\eta=1\}\cup\{|\xi|=+\infty\}.
 \end{split}
\end{align}
Then there exists a positive constant $c_1$ depending on $c, C$ in \eqref{condition} and $T^*$ so that
\begin{align*}
c_1(1-\eta)\leq w(\eta,\xi, t)\leq C(1-\eta)\quad\text{for }\, (\eta,\xi,t)\in [0,1]\times\R\times [0,T^*).
\end{align*}
\end{proposition}

\begin{proof}The proof is inspired by \cite{XZ}.  First of all, we prove that $w\leq C(1-\eta)$. Fix any $t_0<T^*$ and set
$$f=-(\varepsilon+C)(1-\eta)+\varepsilon(1-\eta)^3+w$$ with $\varepsilon$ a sufficiently small positive constant. On $\{t=0\}\cup\{\eta=1\}\cup\{|\xi|=+\infty\},$ we have $$f\leq \big(C-(\varepsilon+C)+\varepsilon\big)(1-\eta)=0.$$
 Let $L=-\partial_t-\eta\partial_\xi+w^2\partial_{\eta\eta}^2$.  Then we have
$$Lf=6w^2\varepsilon(1-\eta)>0.$$
Hence, the maximum of $f$ is not attained in $(0,1)\times\R \times(0,t_0].$ On the other hand,
$$\partial_\eta f\mid_{\eta=0}=C+\varepsilon-3\varepsilon>0.$$ Then the maximum of $f$ is not attained on $\{\eta=0\}.$
In summary, we have $f\leq 0$ when $ (\eta,\xi,t)\in [0,1]\times\R\times[0,T^*).$ Letting $\varepsilon$ goes to $0,$ we get $w\le C(1-\eta)$.\smallskip

Next we prove $w\geq c_1(1-\eta)$. Fix any $t_0<T^*$ and  set
$$f=-e^{-\beta t}\phi \al-\varepsilon(1-\eta)^3+w,$$
where
$$\phi=e^{\frac{\pi}{2}\eta }\sin(\frac{\pi}{2}(1-\eta)),$$
and $\al, \beta$  are positive constants to be determined, and $\varepsilon$ is a sufficiently small positive constant. For $\eta\in[0,1],$ we have
$$1-\eta\leq \sin(\frac{\pi}{2}(1-\eta))\leq \frac{\pi}{2}(1-\eta),$$
which implies
$$\phi\leq e^{\frac{\pi}{2} }\frac{\pi}{2}(1-\eta)=C_2(1-\eta).$$
On $\{t=0\}\cup\{\eta=1\}\cup\{|\xi|=+\infty\},$ we have
$$f\geq \big(c-C_2\al-\varepsilon\big)(1-\eta)\geq0$$
by taking $\al$ small such that $c-C_2\al-\varepsilon\geq0.$
A direct calculation shows
\begin{align*}
    Lf&\leq \Big[-\beta\sin(\frac{\pi}{2}(1-\eta))+C^2(1-\eta)^2(\frac{\pi}{2})^2
    e^{\frac{\pi}{2} }2\cos(\frac{\pi}{2}(1-\eta))\Big]e^{-\beta t}\al+6w^2\varepsilon(\eta-1)\\
    &\leq \Big[-\beta(1-\eta)+C_1(1-\eta)^2\cos(\frac{\pi}{2}(1-\eta))\Big]e^{-\beta t}\al+6w^2\varepsilon(\eta-1)<0
\end{align*}
by taking $\beta$ large depending on $C_1,$ where we have used $w\leq C(1-\eta)$. Hence, the minimum of $f$ is not attained in $(0,1)\times\textbf{R}\times(0,t_0].$ On the other hand,
$$\partial_\eta f\mid_{\eta=0}\leq -e^{-\beta T^*}\frac{\pi}{2}\al+3\varepsilon<0.$$ Then the minimum of $f$ is not attained on $\{\eta=0\}.$ In summary, since $t_0$ is arbitrary, we have $f\ge 0$ when $ (\eta,\xi,t)\in [0,1]\times\textbf{R}\times[0,T^*).$ Letting $\varepsilon\to 0,$ we get
\begin{align*}
w\geq e^{-\beta t}\phi \al\geq e^{-\beta T^*}(1-\eta)\al.
\end{align*}

This completes the proof.
\end{proof}

\section{Global weighted gradient estimate}

\begin{proposition}\label{prop:grad}
Let $u$ be a smooth solution of \eqref{eq:Prandtl} in $[0,T^*)$ and satisfy the following conditions

\begin{align}\label{ph1}
c_0e^{-y}\le \pa_yu(t,x,y)\le C_0e^{-y}\quad \textrm{for }\,\,(t,x,y)\in [0,T^*)\times\R^2_+,
\end{align}
\begin{align}\label{ph2}
|\pa_y^2u(t,x,y)|\le C_0\quad \textrm{for }\,\,(t,x,y)\in [0,T^*)\times\R, \ y\in [0,\delta],
\end{align}
for some $0<c_0<C_0$ and $\delta>0$. Then it holds that
\beno
A(t)\le C_*\quad \text{for any}\,\,t\in \big[0,T^*\big),
\eeno
here $C_*$ is a constant depending  on $C_0, c_0, T^*, \delta$.
\end{proposition}

\begin{remark}
For any $t\in [0,T], T<T^*$, $A(t)$ could be controlled by the energy
$\cE(t)+\int_0^t\cD(s)ds$ introduced in section 6. See Proposition 4.1 in \cite{CWZ}.
\end{remark}

We will use the Crocco transformation introduced in previous section.
By our assumption, there exist two constants depending on $c_0,C_0$ in \eqref{ph1}, which, with abuse of notation, we also denote as $c_0,C_0$, so that
\ben\label{ass:w}
c_0(1-\eta)\le w\le C_0(1-\eta).
\een
To prove the proposition, it suffices to show the following gradient estimates
\ben\label{eq:w-gradient}
|\pa_\eta w|\le C_*,\quad |\pa_\xi w|\le C_*(1-\eta)\quad \text{for}\,\,(\eta,\xi,t)\in (0,1)\times\R\times \big[\f {T^*} 4,T^*\big).
\een
Indeed, we have
\beno
|\pa_y^2u|\le |w\pa_\eta w|\le C_*e^{-y}.
\eeno
And using the formulas
\beno
{\pa_x u}=w\int_0^\eta\f {\pa_\xi w} {w^2}d\eta',\quad \pa_{xy}u=\pa_\eta w\pa_xu+\pa_\xi w,
\eeno
we infer that
\beno
|\pa_xu|+|\pa_{xy}u|\le C_*(1+y)e^{-y}.
\eeno

\subsection{Interior gradient estimate}

First of all, we make $C^\al$ estimate. For this, we need the following proposition from \cite{GI}(see also \cite{WZL}).

\begin{proposition}\label{prop:holder}
Let $f$ be a weak solution of
$$\partial_tf+v\partial_xf=\partial_v(A\partial_vf)+B\partial_vf+S\quad\text{in}\quad Q_{r_0}(z_0),
$$
where $z_0=(x_0,v_0,t_0)$ and
 \begin{align}\label{qr}
  Q_{r}(z_0)=\big\{(x,v,t):|x-x_0-(t-t_0)v_0|<r^3,|v-v_0|<r,t\in(t_0-r^2,t_0]\big\}.
\end{align}
Assume that $0<\lambda\leq A\leq\Lambda, |B|\leq\Lambda$ and $S$ is bounded. Then $f$ is $\alpha$-H\"{o}lder continuous with respect to $(x,v,t)$ in $Q_{r_1}(z_0), r_1<r_0$,
$$
||f||_{C^\alpha(Q_{r_1}(z_0))}\leq C\big(||f||_{L^2(Q_{r_0}(z_0))}+||S||_{L^\infty (Q_{r_0}(z_0))}\big)
$$
for some $\alpha=\alpha(\lambda,\Lambda)$ and $C=C(\lambda,\Lambda,Q_{r_0}(z_0),Q_{r_1}(z_0)).$
\end{proposition}

We consider the regularity of solution  in the domain $[\varepsilon_+,1-\varepsilon_+]\times  \R\times \big[\frac{T^*}{4},T^*\big),$ where $\varepsilon_+$ is a small positive constant. For any $(\eta_0,\xi_0,t_0)\in  [\varepsilon_+,1-\varepsilon_+]\times  \R\times [\frac{T^*}{4},T^*),$ we make the change of variable
\beno
(\bar{\eta},\bar{\xi},\bar{t})=\big(\eta-\eta_0,\xi-\xi_0-\eta_0(t-t_0),t-t_0\big),
\eeno
 and introduce the domain
 \begin{align*}
 \mathcal{ \overline{R}}^0=&\Big\{(\bar{\eta},\bar{\xi},\bar{t})\in\big(-\frac{\varepsilon_+}{100},\frac{\varepsilon_+}{100}\big)\times  (-1,1)\times \big(-\frac{T^*}{100},0\big]\Big\}\\
 =&\Big\{(\eta,\xi,t)\in\big(\eta_0-\frac{\varepsilon_+}{100},\eta_0+\frac{\varepsilon_+}{100}\big)\times  (-1+\xi_0+\eta_0(t-t_0),1+\xi_0+\eta_0(t-t_0))\times \big(t_0-\frac{T^*}{100},t_0\big]\Big\}\\
\subseteq&\big\{(\eta,\xi,t)\in(0,1)\times\R\times (0,T^*) \big\}.
\end{align*}
The equation \eqref{eq:Pran-Cro} is invariant under this transformation, i.e.,
\begin{align}\label{main2dP}
  \partial_{\bar{t}}  w^{-1}+\bar{\eta} \partial_{\bar{\xi}} w^{-1}+\partial_{\bar{\eta}\bar{\eta}}^2 w=0.
\end{align}
Thanks to \eqref{ass:w}, there exist two positive constants $a_1,a_2$ so that
\begin{align}\label{star}
  a_1\leq w\leq a_2\quad\text{for}\ \ (\bar{\eta},\bar{\xi},\bar{t})\in \mathcal{ \overline{R}}^0.
\end{align}
Then for the equation
\begin{align}\label{m12dP}
\partial_{\bar{t}}  w^{-1}+\bar{\eta} \partial_{\bar{\xi}} w^{-1}-\partial_{\bar{\eta}} (w^2\partial_{\bar{\eta}} w^{-1})=0,
  \end{align}
 we use Proposition \ref{prop:holder} in $\mathcal{ \overline{R}}^0$ to obtain a uniform $C^\alpha$ estimate in a smaller domain
 \begin{align*}
 \mathcal{ \overline{R}}^1=\Big\{(\bar{\eta},\bar{\xi},\bar{t})\in\big(-\frac{\varepsilon_+}{200},\frac{\varepsilon_+}{200}\big)\times \big (-\frac{1}{2},\frac{1}{2}\big)\times \big(-\frac{T^*}{200},0\big]\Big\}.
 \end{align*}
Here $C^\alpha$ estimate is independent of choice of $(\eta_0,\xi_0,t_0)$ and only depends on $a_1,a_2,T^*.$\smallskip

 To obtain the gradient estimate, we need to use the localized Schauder estimate from \cite{IM}. We introduce the hypoelliptic H\"older norm $\mathcal{H}^\alpha$ defined by
$$
\|g\|_{\mathcal{H}^\alpha(\mathcal{Q})}
:=\sup_{\mathcal{Q}}|g|+\sup_{\mathcal{Q}}|(\partial_{t}+v\partial_{x})g|+\sup_{\mathcal{Q}}|D_{v}^2g|+[(\partial_{t}+v\partial_{x})g]_{\mathcal{C}^{0,\alpha}(\mathcal{Q})}+[D_{v}^2g]_{\mathcal{C}^{0,\alpha}(\mathcal{Q})},
$$
where $\mathcal{Q}$ denotes a open connected set and $[\cdot]_{\mathcal{C}^{0,\alpha}(\mathcal{Q})}$ denotes the H\"older anisotropic semi- norm (see Definition 2.3 in \cite{IM}), which implies H\"older regularity in usual sense but with lower regularity exponent.

\begin{proposition}\label{localsch}
Given $\alpha\in(0,1)$ and $g\in \mathcal{C}^{0,\alpha}(\R^{2d+1})$ and $a^{i,j},b^i,c\in \mathcal{C}^{0,\alpha}(\R^{2d+1})$ satisfying
\begin{align}\label{(32)}
a^{i,j}(t,x,v)\xi_i\xi_j\geq \lambda |\xi|^2,\quad (t,x,v)\in(0,+\infty)\times \R^{2d},
\end{align}
for some constant $\lambda>0$, it holds that
$$
\|g\|_{\mathcal{H}^\alpha(\mathcal{Q}_1(z_0))}\leq C\|\mathcal{L}g+g\|_{\mathcal{C}^{0,\alpha}(\mathcal{Q}_2(z_0))}
+C\|g\|_{L^\infty(\mathcal{Q}_2(z_0))},
$$
where $z_0=(t_0,x_0,v_0)$ and
$$
\mathcal{Q}_r(z_0)=\big\{(t,x,v):t_0-r^2<t\leq t_0,|x-x_0-(t-t_0)v_0|<r^3,|v-v_0|<r\big\}$$ and $$\mathcal{L}:=\partial_{t}+v\nabla_{x}-a^{i,j}\partial_{v_iv_j}^2-b^i\partial_{v_i}-c,
$$
and $C$ depends on $d,\lambda,\alpha$ and $\|a\|_{\mathcal{C}^{0,\alpha}},
\| b\|_{\mathcal{C}^{0,\alpha}},\|c\|_{\mathcal{C}^{0,\alpha}}.$
\end{proposition}

By the definition of $\|g\|_{\mathcal{H}^\alpha(\mathcal{Q})},$
$\|g\|_{\mathcal{L}^\infty(\mathcal{Q})}$ is controlled by $\|g\|_{\mathcal{H}^\alpha(\mathcal{Q})}$, and thus we can write $(2.18)$ in Lemma 2.12 \cite{IM} as follows
\ben\label{eq:holder-inter}
\|D_{v}g\|_{\mathcal{C}^{0,\alpha}(\mathcal{Q})}\leq C\|g\|_{\mathcal{H}^\alpha(\mathcal{Q})}.
\een

We rewrite \eqref{main2dP} in its equivalent form
\begin{align}\label{regint}
\partial_{\bar{t}} w+\bar{\eta} \partial_{\bar{\xi}} w-w^2\partial^2_{\bar{\eta}\bar{\eta}} w=0,
\end{align}
Let $\mathcal{L}:=\partial_{\bar{t}}+\bar{\eta} \partial_{\bar{\xi}} -w^2\partial^2_{\bar{\eta}\bar{\eta}} $
with $a^{\bar{\eta} \bar{\eta} }=w^2$ satisfying \eqref{(32)} {due to \eqref{star}}.
Then by the interior $C^\alpha$ estimate obtained above, Proposition \ref{localsch} and \eqref{eq:holder-inter}, in a smaller interior domain
$$\mathcal{\overline{R}}^2=\Big\{(\bar{\eta},\bar{\xi},\bar{t})\in\big(-\frac{\varepsilon_+}{400},\frac{\varepsilon_+}{400}\big)\times  \big(-\frac{1}{4},\frac{1}{4}\big)\times \big(-\frac{T^*}{400},0\big]\Big\},
$$
we have
\begin{align*}
 \|\partial_{\bar{\eta}}w\|_{\mathcal{C}^{0,\alpha}(\mathcal{\overline{R}}^2)}+\|w\|_{\mathcal{H}^\alpha(\mathcal{\overline{R}}^2)}\leq C,
\end{align*}
where $C$ is a constant depending only on $a_1,a_2,T^*,\varepsilon_+.$
In particular,
\begin{align}\label{hypohint}
 \|\partial_{\bar{\eta}}w\|_{\mathcal{C}^{0,\alpha}(\mathcal{\overline{R}}^2)}+\| \partial_{\bar{\eta}\bar{\eta}}^2w\|_{\mathcal{C}^{0,\alpha}(\mathcal{\overline{R}}^2)}\leq C.
 \end{align}

 Now we continue to obtain a bound for $\partial_{\bar {\xi}} w$. The method is inspired by Proposition 3.3 in \cite{IM}. While in our case, the equation is quasi-linear, i.e., $a^{\bar{\eta}\bar{\eta}}$ is not a constant but $a^{\bar{\eta}\bar{\eta}}=w^2$. $a^{\bar{\eta}\bar{\eta}}\geq a_1^2$ is a key point in the following proof. We will show that
\begin{align}\label{eq:w-dxi}
 \|\partial_{\bar{\xi}}w\|_{L^\infty(\mathcal{\overline{R}}^3)}\leq C\quad
 \text{in}\,\,\mathcal{\overline{R}}^3,
 \end{align}
where $\mathcal{\overline{R}}^3=\Big\{(\bar{\eta},\bar{\xi},\bar{t})\in(-\frac{\varepsilon_+}{800},\frac{\varepsilon_+}{800})\times  (-\frac{1}{8},\frac{1}{8})\times (-\frac{T^*}{800},0]\Big\}$, and $C$ is a constant depending only on $a_1,a_2,T^*,\varepsilon_+.$\smallskip

We use Bernstein's method to prove \eqref{eq:w-dxi}. Take $0\leq\zeta\in C^\infty$ as a cut-off function supported in $\mathcal{\overline{R}}^2$ and $\zeta=1$ in $\mathcal{\overline{R}}^3$. Let
 $
 L:=-\partial_{\bar{t}}-\bar{\eta} \partial_{\bar{\xi}} +w^2\partial^2_{\bar{\eta}\bar{\eta}}
 $
 and introduce
 $$g=A_1\partial_{\bar{\eta}}w\partial_{\bar{\xi}}w\zeta^9+|\partial_{\bar{\xi}}w|^2\zeta^{12}-A_3\bar{t},
 $$
 with $A_1,A_3$ to be determined later. The choice of powers of $\zeta$ is delicate in order to kill out some terms. By taking $A_1,A_3$ properly, we will prove that $Lg>0$ so that by maximum principle,
 \begin{align}\label{maxg}
 \sup_{\mathcal{\overline{R}}^2} g=\sup_{\partial_p\mathcal{\overline{R}}^2}g.
\end{align}
By a direct calculation, we have
\begin{align*}
 L\partial_{\bar{\xi}}  w&=\partial_{\bar{\xi}} L w-2w\partial_{\bar{\xi}}w\partial^2_{\bar{\eta}\bar{\eta}}w=P_2\partial_{\bar{\xi}}w,\\
 L\partial_{\bar{\eta}}  w&=\partial_{\bar{\eta}} L w+\partial_{\bar{\xi}}w-2w\partial_{\bar{\eta}}w\partial^2_{\bar{\eta}\bar{\eta}}w=P_3+\partial_{\bar{\xi}}w,
 \end{align*}
where we use $P_2,P_3$ to stand for some functions bounded by $C$ through \eqref{hypohint}. Then we get
\begin{align}\begin{split}
  L(\partial_{\bar{\eta}}w\partial_{\bar{\xi}}w \zeta^9)= &
   \zeta^9P_3\partial_{\bar{\xi}}w+ \zeta^9|\partial_{\bar{\xi}}w|^2
   +\zeta^9P_2\partial_{\bar{\xi}}w\partial_{\bar{\eta}}  w\\
   &+L\zeta^9\partial_{\bar{\eta}}w\partial_{\bar{\xi}}w
   +w^2\big(2(\partial^2_{\bar{\eta}\bar{\eta}}w\partial_{\bar{\xi}\bar{\eta}}w) \zeta^9
   +
   2 \partial_{\bar{\eta}}(\partial_{\bar{\eta}}w\partial_{\bar{\xi}}w  ) \partial_{\bar{\eta}}\zeta^9\big)\\
   \geq & \frac{1}{2}\zeta^9|\partial_{\bar{\xi}}w|^2-\epsilon|\partial^2_{\bar{\xi}\bar{\eta}}w|^2 \zeta^{16}-C-\frac{C}{\epsilon},\\
 L( |\partial_{\bar{\xi}}w|^2\zeta^{12})=&w^2\big(2|\partial_{\bar{\xi}\bar{\eta}}^2w|^2\zeta^{12}+4\partial_{\bar{\xi}}w\partial_{\bar{\eta}}\zeta^{12}\partial_{\bar{\xi}\bar{\eta}}^2w
\big)\\&+2\zeta^{12}P_2|\partial_{\bar{\xi}}w|^2
 +|\partial_{\bar{\xi}}w|^2L\zeta^{12}
 \\\geq & \lambda|\partial_{\bar{\xi}\bar{\eta}}^2w|^2\zeta^{12}-C|\partial_{\bar{\xi}}w|^2\zeta^{10}
 -C,
 \end{split}
\end{align}
for some $\lambda>0$ depending only on $a_1$, where we used Young's inequality in the first inequality and \eqref{hypohint}  in both inequalities. Thus, we conclude that
\begin{align}\begin{split}
  Lg\geq& A_1\frac{1}{2}\zeta^9|\partial_{\bar{\xi}}w|^2-\epsilon A_1|\partial^2_{\bar{\xi}\bar{\eta}}w| \zeta^{16}-C A_1\big(1+\frac{1}{\epsilon}\big)\\
 &+ \lambda|\partial_{\bar{\xi}\bar{\eta}}w|^2\zeta^{12}- C|\partial_{\bar{\xi}}w|^2\zeta^{10}+A_3.
 \end{split}
\end{align}
Next take $A_1$ large enough depending on $C,\zeta$. Then take $\epsilon$ small enough so that $\epsilon<\frac{\lambda}{2A_1}$. Finally, taking $A_3$ large enough depending on $A_1,\epsilon,C,$ we obtain $Lg>0.$ Note that $A_1,A_3$ depend only on $\mathcal{\overline{R}}^2,\mathcal{\overline{R}}^3,T^*,a_1,a_2,\varepsilon_+$. Thanks to
$$
\partial_{\bar{\eta}}w\partial_{\bar{\xi}}w \leq C_{\varepsilon }|\partial_{\bar{\eta}}w|^2+\varepsilon |\partial_{\bar{\xi}}w|^2,
$$ where $\varepsilon $ is any small positive constant, by \eqref{hypohint} and \eqref{maxg}, in $\mathcal{\overline{R}}^2,$
$$
 -C-A_3\bar{t}+\frac{1}{2}|\partial_{\bar{\xi}}w|^2\zeta^{12}\leq g\leq C.
$$
This completes the proof.
\smallskip

By \eqref{hypohint} and \eqref{eq:w-dxi}, especially at $(\bar{\eta},\bar{\xi},\bar{t})=(0,0,0),$ we have $|{\nabla}_{\bar \xi,\bar\eta} w|(0,0,0)\leq C_{a_1,a_2,T^*,\varepsilon_+}.$
Back to the original coordinate $(\eta,\xi,t)$, we have $|\nabla_{\xi,\eta} w|(\eta_0,\xi_0,t_0)\leq C_{a_1,a_2,T^*,\varepsilon_+}.$ Since the bound is independent of choice of $(\eta_0,\xi_0,t_0)$ and $(\eta_0,\xi_0,t_0)$ is arbitrary, we obtain
\begin{align}\label{eq:w-gradient-i}
|\nabla_{\xi,\eta} w| _{L^\infty([\varepsilon_+,1-\varepsilon_+]\times  \R\times [\frac{T^*}{4},T^*))}\leq C_{a_1,a_2,T^*,\varepsilon_+}.
\end{align}

\subsection{Gradient estimate near $\eta=1$}
Now we make the boundary gradient estimate for $w$ in the domain $\big\{(x,y,\tau) \in \R^2_+\times [\frac{2T^*}{3},T^*)|1-\varepsilon_+\leq u\big\},$ or equivalently $\big\{(\eta,\xi,t)\in[1-\varepsilon_+,1]\times  \R\times [\frac{2T^*}{3},T^*)\big\}$, where $\varepsilon_+$ is chosen so that
$1-5\varepsilon_+>\frac{1}{2}$. Set
\beno
&&D_1=\Big\{(\eta,\xi,t)\in[1-4\varepsilon_+,1]\times  \R\times [\frac{T^*}{5},T^*)\Big\},\\
 &&D_2=\Big\{(\eta,\xi,t)\in[1-3\varepsilon_+,1]\times  \R\times [\frac{T^*}{4},T^*)\Big\}.
\eeno
 For any $x_0=(\eta_0,\xi_0,t_0)\in D_2,$ we define a reversible transform by
\begin{align}\label{Ucons}
    (\overline{\eta},\overline{\xi},\overline{t})=\Big(\frac{\eta
    -\eta_0}{ 1-\eta_0},\frac{\xi-\xi_0-\eta_0(t-t_0)}{ 1-\eta_0},
   t-t_0\Big),
\end{align}
and introduce a function
$$h=w^{-1}(1-\eta_0).$$
Then we have
\begin{align}\label{bart}
    \partial_\eta=\frac{1}{ 1-\eta_0}\partial_{\bar{\eta}},\,\,
   \partial_\xi=\frac{1}{ 1-\eta_0}\partial_{\bar{\xi}} ,\,\,
   \partial_t=-\frac{\eta_0}{ 1-\eta_0}\partial_{\bar{\xi}} +\partial_{\bar{t}}, \,\, \partial_{\bar{t}}=\partial_{t}+\eta_0\partial_{\xi}.
\end{align}
Under the coordinates $(\overline{\eta},\overline{\xi},\overline{t}),$ $x_0=(0,0,0)$ and the equation \eqref{eq:Pran-Cro} takes
\begin{align}\label{1}
\partial_{\bar{t}}  h+\bar{\eta} \partial_{\bar{\xi} } h-\partial_{\bar{\eta}}\Big(\frac{w^2}{(1-\eta_0)^2}\partial_{\bar{\eta}} h\Big)=0,
\end{align}
or equivalently,
\begin{align}\label{inf2}
\partial_{\bar{t}}   \frac{w}{1-\eta_0}+\bar{\eta} \partial_{\bar{\xi} } \frac{w}{1-\eta_0}-\frac{w^2}{(1-\eta_0)^2}\partial^2_{\bar{\eta}\bar{\eta}} \frac{w}{1-\eta_0}=0.
\end{align}

As before, we can choose a small enough constant $r_0$ depending on $T^*,$ but independent of $x_0$ such that
$$\big\{(\bar{\eta},\bar{\xi},\bar{t})\in {Q}_{r_0}\big\}\subset D_1,$$
where ${Q}_r$ is a cube defined in \eqref{qr} with $z_0=0$. By \eqref{ass:w}, $0<\lambda\leq\frac{w^2}{(1-\eta_0)^2},h\leq \Lambda$ in $Q_{r_0}$, where $\lambda,\Lambda$ are independent of $\eta_0$ and determined by $c_0,C_0$ in \eqref{ass:w}. Then by Proposition \ref{prop:holder}, we have
$$
 \|h\|_{C_{\bar{\eta},\bar{\xi},\bar{t}}^\alpha({Q}_{\frac{r_0}{2}})} \leq C_{\lambda,\Lambda,r_0}.
$$
This implies that
$$
  \Big\|\frac{w}{1-\eta_0}\Big\|_{C_{\bar{\eta},\bar{\xi},\bar{t}}^\alpha({Q}_{\frac{r_0}{2}})} \leq C_{\lambda,\Lambda,r_0},
$$
by noting that for any $x,y\in {Q}_{r_0},$
$$
 \Big|\frac{w}{1-\eta_0}(x)- \frac{w}{1-\eta_0}(y)\Big|=\Big| (h(x)-h(y) ) \frac{w(x)w(y)}{(1-\eta_0)^2}\Big|\leq C|h(x)-h(y)|.
$$

Next we consider $\frac{w}{1-\eta_0}$ as a whole in \eqref{inf2}. By the interior estimates similar as in subsection 5.1, we have
\begin{align}\label{eta1est}
\begin{split}
  &\big\|\partial_{\bar{\eta}} \frac{w}{1-\eta_0}\big\|_{L^\infty({Q}_{\frac{r_0}{4}})}\leq C,\ \big\|\partial_{\bar{t}} \frac{w}{1-\eta_0}\big\|_{L^\infty({Q}_{\frac{r_0}{4}})}+\big\|\partial_{\bar{\xi}} \frac{w}{1-\eta_0}\|_{L^\infty({Q}_{\frac{r_0}{4}})}
   \leq C.\end{split}
\end{align}
Transforming back to the original coordinate $(\eta,\xi,t)$, we get by \eqref{bart}  that
\begin{align*}
& |\partial_{t}+\eta_0\partial_{\xi}w(x_0)|\leq (1-\eta_0)C_{\lambda,\Lambda,r_0},\\
&|\partial_{\eta}w(x_0)|+|\partial_{\xi}w(x_0)|+|\partial_{t}w(x_0)| \leq C_{\lambda,\Lambda,r_0}.
\end{align*}
Since the choice of $x_0$ is arbitrary and $C_{\lambda,\Lambda,r_0}$ is independent of $x_0$, we have
\begin{align}\label{1drvgr}
 &\|\partial_{\eta}w\|_{L^\infty(D_2)}+\|\partial_{\xi}w\|_{L^\infty(D_2)}+\|\partial_{t}w\|_{L^\infty(D_2)}\leq C_{\lambda,\Lambda,T^*},\\
&\|\partial_{t}w+\eta\partial_{\xi }w\|_{L^\infty(D_2)}\leq C_{\lambda,\Lambda,T^*}(1-\eta).
\end{align}
In particular, by \eqref{1drvgr},
\begin{align}\label{L}
\|\partial_{t}w+\partial_{\xi }w\|_{L^\infty(D_2)}\leq C_{\lambda,\Lambda,T^*}(1-\eta).
\end{align}

To prove \eqref{eq:w-gradient}, we need to refine the estimate for $\partial_{\xi}w$, i.e.,
\begin{align}
\|\partial_{\xi}w\|_{L^\infty([1-2\varepsilon_+,1]\times\R\times [\frac{2T^*}{3},T^*))}\leq C_{\lambda,\Lambda,T^*}(1-\eta),
\end{align}
which will be proved by the following two lemmas.

\begin{lemma}\label{9.1}
For any $(\eta_0,\xi_0,t_0)\in [1-2\varepsilon_+,1]\times  \R\times [\frac{T^*}{2},T^*),$ we have
\begin{align}\label{12}
    \Big|\frac{w(\eta_0,\xi,t_0)-w(\eta_0,\xi_0,t_0)}{1-\eta_0}\Big|\leq C(\xi-\xi_0)^{\frac{1}{2}},\,\,\,\xi\in\big[\xi_0,\xi_0+\frac{(1-\eta_0)^2}{2}\big].
\end{align}
 where $C$ is independent of the choice of $(\eta_0,\xi_0,t_0)$ and depends only on $T^*, \lambda, \Lambda$.
\end{lemma}

\begin{proof}
We introduce the domain
$$\Omega_{\xi_0,t_0}=\big\{(\xi,t)\in(\xi_0,\xi_0+1)\times(t_0-c^*,t_0]\big\},
$$
where $c^*=\frac{T^*}{20}.$ Then $[1-2\varepsilon_+,1]\times \Omega_{\xi_0,t_0}\subset D_2.$ Set
$L=-\partial_t-\partial_\xi$ and introduce a nonnegative smooth function
$$ \varphi(\xi)=\left\{
\begin{aligned}
0\quad&\xi_0\leq\xi\leq\xi_0+\frac{ (1-\eta_0)^2}{2},\\
1\quad&\xi_0+ (1-\eta_0)^2\leq\xi\leq\xi_0+1
\end{aligned}
\right.
$$
with $\partial_{\xi}\varphi\geq 0$ and
\begin{align*}
   g_1=&\frac{w(\eta_0,\xi,t)-w(\eta_0,\xi_0,t_0)}{1-\eta_0}-\Big[A_1(\xi-\xi_0)^{\frac{1}{2}}
   +\frac{C_1}{1-\eta_0}(t_0-t)+C_2(t_0-t)\Big] \\&-\frac{A_2}{1-\eta_0}\varphi(\xi)(\xi-\xi_0)\quad\,\,\text{in} \,\,\Omega_{\xi_0,t_0},\\
      g_2=&\frac{w(\eta_0,\xi,t)-w(\eta_0,\xi_0,t_0)}{1-\eta_0}+\Big[A_1(\xi-\xi_0)^{\frac{1}{2}}
   +\frac{C_1}{1-\eta_0}(t_0-t)+C_2(t_0-t)\Big]\\&+\frac{A_2}{1-\eta_0}\varphi(\xi)(\xi-\xi_0)\quad
    \,\,\text{in}\,\,\Omega_{\xi_0,t_0},
\end{align*}
where $A_1,C_1,A_2,C_2$ are big constants to be determined. As $\xi-\xi_0>0 \,\,in \,\,\Omega_{\xi_0,t_0},$ $g_1,g_2$ are well defined.
We only prove $g_1\leq 0\,\,in \,\,\Omega_{\xi_0,t_0}$ and the argument for $g_2\geq 0\,\,in \,\,\Omega_{\xi_0,t_0}$ is similar. Then taking $t=t_0$,  and thanks to $\varphi(\xi)=0\,\,in\,\,\big[\xi_0,\xi_0+\frac{(1-\eta_0)^2}{2}\big],$ we have \eqref{12}.\smallskip

We will first show that
$g_1\leq 0$ on $\partial_p\Omega_{\xi_0,t_0},$ where
$$
\partial_p\Omega_{\xi_0,t_0}=\big\{\xi=\xi_0+1,t\in[t_0-c^*,t_0]\}\cup\big\{t=t_0-c^*,\xi\in[\xi_0,\xi_0+1]\big\}
\cup\big\{\xi=\xi_0,t\in[t_0-c^*,t_0]\big\}.
$$
Since $\|\frac{w(\eta_0,\xi,t)-w(\eta_0,\xi_0,t_0)}{1-\eta_0}\|_{L^\infty(\Omega_{\xi_0,t_0})}\leq C$ by \eqref{ass:w}, we take $C_2,A_2\geq C$ to ensure $g_1\leq 0$ on $\partial_p\Omega_{\xi_0,t_0}\cap\big(\{\xi=\xi_0+1\}\cup\{t=t_0-c^*\}\big).$ By \eqref{1drvgr}, we take $C_1$ big enough to ensure $g_1\leq 0$ on $\partial_p\Omega_{\xi_0,t_0}\cap\{\xi=\xi_0\}.$
Next we will show that $Lg_1>0$ in $\Omega_{\xi_0,t_0}$, and therefore $g_1$ can only obtain its maximum at $\partial_p\Omega_{\xi_0,t_0}$ so that we can conclude $g_1\leq 0.$

Firstly, by \eqref{L}, we have
$$\Big|L \frac{w(\eta_0,\xi,t)-w(\eta_0,\xi_0,t_0)}{1-\eta_0}\Big|\leq C.$$
Secondly,
$$\Big|L\Big(\frac{C_1}{1-\eta_0}(t_0-t)+C_2(t_0-t)\Big)\Big|\leq \frac{C}{1-\eta_0}+C.$$
Since we already fix $C_1,C_2$ for the boundary condition as above, we only write $C$ in the right hand side now.
Thirdly,
$$ L\big[-(\xi-\xi_0)^{\frac{1}{2}}\big]=\partial_\xi(\xi-\xi_0)^{\frac{1}{2}}\geq\left\{
\begin{aligned}
&\frac{1}{2}(1-\eta_0)^{-1}\,\,\,&\xi_0\leq\xi\leq\xi_0+(1-\eta_0)^2,\\
& 0\,\,\,&\xi_0+(1-\eta_0)^2<\xi\leq \xi_0+1.
\end{aligned}
\right.
$$
For $ \varphi(\xi)$, we note that $ -\partial_\xi(-\varphi(\xi))\geq 0.$ Hence,
$$L\Big(-\frac{1}{1-\eta_0}\varphi(\xi)(\xi-\xi_0)\Big)=\partial_\xi\Big(\frac{1}{1-\eta_0}\varphi(\xi)(\xi-\xi_0)\Big)\geq \frac{1}{1-\eta_0}\varphi(\xi)$$
in $\Omega_{\xi_0,t_0},$
which implies that
$$L\Big(-\frac{1}{1-\eta_0}\varphi(\xi)(\xi-\xi_0)\Big)\geq\left\{
\begin{aligned}
&\frac{1}{1-\eta_0}\,\,\,&\xi_0+(1-\eta_0)^2\leq\xi\leq\xi_0+1,\\
&0\,\,\,&\xi_0\leq\xi<\xi_0+(1-\eta_0)^2.
\end{aligned}
\right.
$$
Therefore, by taking $A_1,A_2$ large independent of $\eta_0,$ we have
$$L\Big(-A_1(\xi-\xi_0)^{\frac{1}{2}}
   -\frac{A_2}{1-\eta_0}\varphi(\xi)(\xi-\xi_0)\Big)> \frac{C}{1-\eta_0}$$
for $\xi_0\leq\xi\leq\xi_0+1$, which implies $Lg_1>0.$
\end{proof}

Now we further improve the result in Lemma \ref{9.1}.

\begin{lemma}\label{thmxirf}
For any $(\eta_0,\xi_0,t_0)\in [1-2\varepsilon_+,1]\times\R\times [\frac{2T^*}{3},T^*),$ we have
\begin{align}\label{xi1}
\Big|\frac{w(\eta_0,\xi,t_0)-w(\eta_0,\xi_0,t_0)}{1-\eta_0}\Big|\leq C_{\lambda,\Lambda,T^*}(\xi-\xi_0),\,\,\,\xi\in \big[\xi_0,\xi_0+\frac{c(1-\eta_0)^2}{8\sqrt{2}}\big],
\end{align}
where $C_{\lambda,\Lambda,T^*}$ is independent of choice of $(\eta_0,\xi_0,t_0)$ and $c<1$ only depends on $T^*.$ In particular, letting $\xi_0\rightarrow\xi_0^+,$ we obtain
\begin{align}
 \Big|\frac{\partial_\xi w(\eta_0,\xi_0,t_0)}{1-\eta_0}\Big|\leq C_{\lambda,\Lambda,T^*}.
\end{align}
\end{lemma}

\begin{proof}
We introduce the domain
$$\mathcal{R}_{\xi_0,t_0}=\big\{(\xi,t)|t\in(t_0-\varepsilon,t_0],\xi\in(\xi_0+t-t_0,\xi_0+t-t_0+\varepsilon^2)\big\},
$$
where $\varepsilon=\sqrt{\frac{c(1-\eta_0)^2}{2}}$ and we take $c$ small only depending on $T^*$ and hence independent of $\eta_0$ such that $\frac{2T^*}{3}-\varepsilon>\frac{T^*}{2}.$
Set $L=-\partial_t-\partial_\xi$ and introduce a nonnegative smooth function
$$ \varphi(s)=\left\{
\begin{aligned}
0\quad& 0\leq s\leq\frac{\varepsilon^2}{4\sqrt{2}},\\
1\quad& \frac{\varepsilon^2}{2\sqrt{2}}\leq s\leq\frac{\varepsilon^2}{\sqrt{2}},
\end{aligned}
\right.
$$
and
\begin{align*}
   g_1=&\frac{w(\eta_0,\xi,t)-w(\eta_0,\xi_0,t_0)}{1-\eta_0}-\big[A_1(\xi-\xi_0)+C_1(t_0-t)\big]\\&-\frac{C_2}{1-\eta_0}(\xi-(\xi_0+t-t_0))\varphi\Big(\frac{\xi-t-(\xi_0-t_0)}{\sqrt{2}}\Big)\quad\,\,\text{in}\,\,\mathcal{R}_{\xi_0,t_0},\\
     g_2=&\frac{w(\eta_0,\xi,t)-w(\eta_0,\xi_0,t_0)}{1-\eta_0}+\big[A_1(\xi-\xi_0)+C_1(t_0-t)\big]\\&+\frac{C_2}{1-\eta_0}(\xi-(\xi_0+t-t_0))\varphi\Big(\frac{\xi-t-(\xi_0-t_0)}{\sqrt{2}}\Big)
\quad\,\,\text{in}\,\,\mathcal{R}_{\xi_0,t_0},
\end{align*}
where $A_1,C_1,C_2$ are big constants to be determined.
We only prove $g_1\leq 0\,\,in\,\,\mathcal{R}_{\xi_0,t_0}$, since the argument for $g_2\geq 0\,\,in\,\,\mathcal{R}_{\xi_0,t_0}$ is similar. After we prove these, take $t=t_0$ and consider $\xi\in[\xi_0,\xi_0+\frac{c(1-\eta_0)^2}{8\sqrt{2}}].$
Thanks to $|\frac{\xi-\xi_0}{\sqrt{2}}|\leq\frac{c(1-\eta_0)^2}{16}\leq\frac{\varepsilon^2}{8},$
we have
$$
\varphi\Big(\frac{\xi-t_0-(\xi_0-t_0)}{\sqrt{2}}\Big)=0\,\,\text{in}\,\,[\xi_0,\xi_0+\frac{c(1-\eta_0)^2}{8\sqrt{2}}].
$$
Then we have \eqref{xi1}. \smallskip

First of all, we  show
$g_1\leq 0$ on $\partial_p\mathcal{R}_{\xi_0,t_0}.$ On
$\big\{(\xi,t)|t\in[t_0-\varepsilon,t_0],\xi=\xi_0+t-t_0\big\},$ we have
\begin{align*}
&\frac{w(\eta_0,\xi_0+t-t_0,t)-w(\eta_0,\xi_0,t_0)}{1-\eta_0}\\
&=\int_{0}^{t-t_0}\frac{1}{1-\eta_0}\frac{d}{ds} w(\eta_0,\xi_0+s,t_0+s)ds\\
&=\int_{0}^{t-t_0}\frac{1}{1-\eta_0} (\partial_\xi+\partial_t) w(\eta_0,\xi_0+s,t_0+s)ds.
\end{align*}
Hence, by \eqref{L}, we have
\begin{align}\label{sli}
\Big|\frac{w(\eta_0,\xi_0+t-t_0,t)-w(\eta_0,\xi_0,t_0)}{1-\eta_0}\Big|\leq C(t_0-t).
\end{align}
On $\big\{(\xi,t)|t\in[t_0-\varepsilon,t_0],\xi=\xi_0+t-t_0+\varepsilon^2\big\}$,
we have
\begin{align*}
&\Big|\frac{w(\eta_0,\xi_0+t-t_0+\varepsilon^2,t)-w(\eta_0,\xi_0,t_0)}{1-\eta_0}\Big|\\
&\leq\Big|\frac{w(\eta_0,\xi_0+t-t_0+\varepsilon^2,t)-w(\eta_0,\xi_0+t-t_0,t)}{1-\eta_0}\Big|+\Big|\frac{w(\eta_0,\xi_0+t-t_0,t)-w(\eta_0,\xi_0,t_0)}{1-\eta_0}\Big|\\
&\leq  \frac{C}{1-\eta_0}(\xi-(\xi_0+t-t_0))\varphi\Big(\frac{\xi-t-(\xi_0-t_0)}{\sqrt{2}}\Big)\mid_{\xi=\xi_0+t-t_0+\varepsilon^2}
+C(t_0-t),
\end{align*}
where we have used \eqref{sli} and \eqref{1drvgr} in the last inequality.
On $\big\{(\xi,t)|t=t_0-\varepsilon,\xi\in[\xi_0-\varepsilon,\xi_0-\varepsilon+\varepsilon^2]\big\},$ we have
\begin{align*}
&\Big|\frac{w(\eta_0,\xi,t_0-\varepsilon)-w(\eta_0,\xi_0,t_0)}{1-\eta_0}\Big|
\\&\leq \Big|\frac{w(\eta_0,\xi,t_0-\varepsilon)-w(\eta_0,\xi_0-\varepsilon,t_0-\varepsilon)}{1-\eta_0}\Big|+\Big|\frac{w(\eta_0,\xi_0-\varepsilon,t_0-\varepsilon)-w(\eta_0,\xi_0,t_0)}{1-\eta_0}\Big|\\&\leq C\varepsilon\leq C(t_0-t)\mid_{t=t_0-\varepsilon},
\end{align*}
where we used \eqref{sli} with $t=t_0-\varepsilon$ for the second term and   used \eqref{12} for the first term(since $\xi\in[\xi_0-\varepsilon,\xi_0-\varepsilon+\varepsilon^2]$ with $\varepsilon^2=\frac{c(1-\eta_0)^2}{2}\leq\frac{(1-\eta_0)^2}{2}$). In summary, taking $C_1,C_2$ large, we have $g_1\leq 0$ on $\partial_p\mathcal{R}_{\xi_0,t_0}.$\smallskip

Next we show $Lg_1>0$ in $\mathcal{R}_{\xi_0,t_0}$ by taking $A_1$ large. Then $g_1$ can only obtain its maximum on $\partial_p\mathcal{R}_{\xi_0,t_0}$
 so that we can conclude $g_1\leq 0.$ By \eqref{L}, we have
$$
L\Big(\frac{w(\eta_0,\xi,t)-w(\eta_0,\xi_0,t_0)}{1-\eta_0}-C_1(t_0-t)\Big)\geq-C.
$$
Since we already fix $C_1$ for the boundary condition as above, we only write $C$ in the right hand side now. Using the facts that
\begin{align*}
& (\partial_t+\partial_\xi)\varphi\Big(\frac{\xi-t-(\xi_0-t_0)}{\sqrt{2}}\Big)=
    \frac{1}{\sqrt{2}}\partial_s\varphi\mid_{s=\frac{\xi-t-(\xi_0-t_0)}{\sqrt{2}}}(1-1)=0,\\
& (\partial_t+\partial_\xi)(\xi-(\xi_0+t-t_0))=0,
\end{align*}
we infer that
$$
(\partial_t+\partial_\xi)\frac{C_2}{1-\eta_0}\big(\xi-(\xi_0+t-t_0)\big)\varphi\Big(\frac{\xi-t-(\xi_0-t_0)}{\sqrt{2}}\Big)=0.
$$
In summary, we obtain
$$Lg_1\geq A_1-C.$$
By taking $A_1$ large, we have $Lg_1>0.$
\end{proof}

\subsection{Gradient estimate near $\eta=0$}
Thanks to $\eta=u,$ we have
\begin{align}\label{A3dm}
 \big\{(\eta,\xi,t) \in(0,c(1-e^{-\delta}))\times  \R\times [0,T^*)\big\}
\subset\big\{(y,x,\tau) \in(0,\delta)\times  \R\times [0,T^*)\big\}.
\end{align}
Let $\epsilon_1=\frac{c(1-e^{-\delta})}{2}$ and introduce the domain
$$
D_0=\big\{(\eta,\xi,t) \in(0,\epsilon_1)\times  \R\times [0,T^*)\big\}.
$$
Without loss of generality, we assume $\epsilon_1$ small so that
\begin{align}\label{ep1}
\epsilon_1\leq\min\big\{\frac{1}{4},\frac{T^*}{8}\big\}.
\end{align}
By \eqref{ass:w}, we have
\begin{align}\label{aA}
    a \leq w\leq A\quad \text{in}\,\,D_0
\end{align}
 for some positive constants $a,A.$ By our assumption \eqref{ph2},  we have
\begin{align}\label{bdetaw}
 |\partial_\eta w|\leq C\quad \text{in}\,\,D_0,
\end{align}
{where $C$ depends on constants $a,A$ and constant $C_0$ in \eqref{ph2}. It remains to prove that
\begin{align}
\|\partial_{\xi}w\|_{L^\infty((0,\frac{\epsilon_1}{2})\times  \R\times [\frac{2T^*}{3},T^*))}\leq C,
\end{align}
which will be proved in the following two lemmas.

\begin{lemma} \label{xi0el}
It holds that
\begin{align}\label{2.8}
\|\partial_{\xi}w\|_{L^\infty((0,\frac{3}{4}\epsilon_1)\times  \R\times [\frac{T^*}{8},T^*))}\leq C\eta^{-3},
\end{align}
where $C$ depends only on $a,A,T^*.$
\end{lemma}

\begin{proof}
For any $x_0=(\eta_0,\xi_0,t_0)\in(0,\frac{3}{4}\epsilon_1)\times\R\times[\frac{T^*}{8},T^*),$ we introduce a reversible transform $T_{x_0}:\mathcal{Q}_1=(-1,1)\times(-1,1)\times(-1,0]\mapsto D_0$ such that
\begin{align*}
(\bar{\eta},\bar{\xi},\bar{t})=\Big(\frac{\eta-\eta_0}{\frac{\eta_0}{3}},\frac{\xi-\xi_0-\eta_0(t-t_0)}{(\frac{\eta_0}{3})^3},\frac{t-t_0}{(\frac{\eta_0}{3})^2}\Big),
\end{align*}
where $(\eta,\xi,t)=T_{x_0}(\bar{\eta},\bar{\xi},\bar{t}).$
Next let $$\overline{w}(\bar{\eta},\bar{\xi},\bar{t})=
w(\eta,\xi,t)$$ in the domain $\mathcal{Q}_1.$
Then $\overline{w}$ satisfies
$$
-\partial_{\bar{t}}\bar{w}-\bar{\eta} \partial_{\bar{\xi}} \overline{w} +\overline{w}^2\partial^2_{\bar{\eta}\bar{\eta}}\overline{w}=0.
$$
By \eqref{aA}, $a\leq\overline{w}\leq A.$
 Applying the interior estimate to $\overline{w}$ in $\mathcal{Q}_1$, we get $$|\partial_{\bar{\xi}}\overline{w}(0,0,0)|\leq C$$ where $C$ only depends on $a, A.$ Back to the original coordinate, we have $|\partial_{\xi}w(x_0)\eta_0^{3}|\leq C.$ Since $C$ is independent of the choice of $x_0,$ we have
$
 |\partial_{\xi}w(\eta,\xi,t)\eta^{3}|\leq C.
$
\end{proof}

\begin{lemma}\label{xi0e}
It holds that
\begin{align}
\|\partial_{\xi}w\|_{L^\infty((0,\frac{\epsilon_1}{2})\times  \R\times [\frac{2T^*}{3},T^*))}\leq C,
\end{align} where $C$ depends only on $a,A,T^*$ and $\epsilon_1.$
\end{lemma}
\begin{proof}For any $x_0=(\eta_0,\xi_0,t_0)\in(0,\frac{\epsilon_1}{2})\times  \R\times [\frac{2T^*}{3},T^*),$ we introduce a reversible transform $T_{\xi_0,t_0}:\mathcal{Q}=(0,\frac{\epsilon_1}{2})\times(-1,1)\times(-\frac{T^*}{2},0]\mapsto (0,\frac{3}{4}\epsilon_1)\times  \R\times [\frac{T^*}{8},T^*)$ such that
\begin{align*}
(\bar{\eta},\bar{\xi},\bar{t})=(\eta,\xi-\xi_0,t-t_0).
\end{align*} Then in $\mathcal{Q},$ \eqref{aA} and \eqref{bdetaw}( i.e.  $
   |\partial_{\bar{\eta}} w|\leq C
$) hold, and we have
$$-\partial_{\bar{t}}w-\bar{\eta} \partial_{\bar{\xi}} w +w^2\partial^2_{\bar{\eta}\bar{\eta}}w=0.$$
We take a cut-off function $\zeta(\bar{\xi},\bar{t})\in C^\infty\big([-1,1]\times[-\frac{T^*}{2},0]\big)$, which satisfies $0\leq\zeta\leq 1 $ and $$\zeta(\bar{\xi},\bar{t})=\left\{
\begin{aligned}
0\qquad& \Big\{\bar{\xi}\in [-1,-\frac{3}{4}]\cup[\frac{3}{4},1]\Big\}\cup\Big\{\bar{t}\in[-\frac{T^*}{2},-\frac{3T^*}{8}]\Big\},\\
1\qquad& \big[-\frac{1}{2},\frac{1}{2}\big]\times\big[-\frac{T^*}{4},0\big].
\end{aligned}
\right.
$$
Let $L=-\partial_{\bar{t}}-\bar{\eta}\partial_{\bar{\xi}}+w^2\partial_{\bar{\eta}\bar{\eta}}^2$ and introduce
\begin{align*}
   f= &|\partial_{\bar{\xi}}  w|^2\bar{\eta}^2 \zeta^{10}
  +M_1(|\partial_{\bar{\xi}}  w|^2+1)^{\frac{2}{3}} \zeta^{10}+M_2(|\partial_{\bar{\xi}}  w|^2+B)^{\frac{1}{3}}|\partial_{\bar{\eta}}  w|^2 \zeta^{10}\\&+M_3|\partial_{\bar{\eta}}  w|^2\zeta^{10}-M_4\bar{t}+\bar{\eta},
\end{align*} where $B,M_1,M_2,M_3,M_4,$ are big positive constants to be determined.\medskip

\no(A) First of all, we consider $\partial_{p}\mathcal{Q}\setminus\{\bar{\eta}=0\}.$ By interior estimates in the previous section, we have $f\leq C$ on $\{\bar{\eta}=\frac{\epsilon_1}{2}\},$ where $C$ depends only on $a,A,T^*,M_i$ and $\epsilon_1.$ By the definition of $\zeta,$ we have $\max_{\partial_{p}\mathcal{Q}\setminus\{\bar{\eta}=0,\bar{\eta}=\frac{\epsilon_1}{2}\}}f\leq C,$ where $C$ depends only on $T^*$ and  $M_4$.\smallskip

\no(B) Next we will prove that $\max_{\mathcal{\bar{Q}}} f$ is not attained on the boundary $\{\bar{\eta}=0\}.$ Thanks to $\partial_{\bar{\eta}} w\mid_{\bar{\eta}=0}=0$, we have $\partial_{\bar{\eta}\bar{\xi}}^2 w\mid_{\bar{\eta}=0}=0.$ Therefore,
\begin{align*}
   \partial_{\bar{\eta}}f\mid _{\bar{\eta}=0}=1>0.
\end{align*}
So, $\max_{\mathcal{\bar{Q}}} f$ is not attained on the boundary $\{\bar{\eta}=0\}.$\smallskip

\no(C) We will show that $\max_{\mathcal{Q}} f$ is not attained in $\mathcal{Q}$ by proving that
$$
Lf>0\quad \quad \text{in} \quad\mathcal{Q}.
$$

In what follows, we denote by $\beta_i(i=1,\cdots 10)$ some small constants determined later, and by $\gamma_i(i=1,2,3,4)$ some constants depending only on $a$, and by $C$ a constant depending on $a, A, T^*, \beta_i(i=1,\cdots, 10), M_1, M_2, M_3, B$, and by $C_{\beta_i}$ a constant depending only on $\beta_i, a, A, T^*$.

By direct calculations, we get
 \begin{align*}
&L(g_1g_2)=L(g_1)g_2+L(g_2)g_1+2w^2\partial_{\bar{\eta}}g_1\partial_{\bar{\eta}}g_2,\\
&L\partial_{\bar{\xi}}  w
=-2w\partial_{\bar{\xi}}w\partial_{\bar{\eta}\bar{\eta}} ^2 w  ,\quad L\partial_{\bar{\eta}}  w
=-2w\partial_{\bar{\eta}}w\partial_{\bar{\eta}\bar{\eta}} ^2 w +\partial_{\bar{\xi}}  w.
\end{align*}
Then we have
\begin{align*}
    L|\partial_{\bar{\xi}}  w|^2
&=-4w(\partial_{\bar{\xi}}w)^2\partial^2_{\bar{\eta}\bar{\eta}}  w +2w^2|\partial_{\bar{\xi}\bar{\eta}}^2w|^2
\\ L|\partial_{\bar{\eta}}  w|^2
&=2\partial_{\bar{\xi}}w\partial_{\bar{\eta}}  w-4w(\partial_{\bar{\eta}}w)^2\partial^2_{\bar{\eta}\bar{\eta}}  w +2w^2|\partial_{\bar{\eta}\bar{\eta}}^2w|^2,\\
 L(|\partial_{\bar{\xi}}  w|^2\bar{\eta}^2)
&=-4w(\partial_{\bar{\xi}}w)^2\bar{\eta}^2\partial^2_{\bar{\eta}\bar{\eta}}  w +2w^2|\partial_{\bar{\xi}\bar{\eta}}^2w|^2\bar{\eta}^2+2w^2|\partial_{\bar{\xi}}  w|^2+2w^22\bar{\eta}2\partial_{\bar{\xi}}  w\partial_{\bar{\xi}\bar{\eta}}^2w.
\end{align*}
By \eqref{2.8}, we have
\begin{align}\label{eq:w-dxi-rough}
 |\partial_{\bar{\xi}}w\bar{\eta}^3|\leq C  \Rightarrow
  |\partial_{\bar{\xi}}w|^{\frac{2}{3}}\bar{\eta}^2\leq C,\quad |\partial_{\bar{\xi}}w|^{\frac{1}{3}}\bar{\eta}\leq C .
\end{align}
Then by Young's inequality, we get
\begin{align}\label{beta1}
 L(|\partial_{\bar{\xi}}  w|^2\bar{\eta}^2)
 \geq 2w^2|\partial_{\bar{\xi}}w|^2-\beta_1 |\partial_{\bar{\xi}}w|^2- C_{\beta_1}|\partial_{\bar{\xi}}w|^{\frac{2}{3}}|\partial^2_{\bar{\eta}\bar{\eta}}  w |^2-   C_{\beta_1}\frac{|\partial^2_{\bar{\xi}\bar{\eta}}  w |^2}{|\partial_{\bar{\xi}}w|^{\frac{2}{3}}+1}-C.
\end{align}
Then by $\partial_{\bar{\eta}} \zeta=0$ and \eqref{eq:w-dxi-rough}, taking $\beta_1$ small depending only on $a$,  we have
\begin{align}\begin{split}\label{split}
 L((|\partial_{\bar{\xi}}  w|^2\bar{\eta}^2)\zeta^{10})
\geq &-  C_{\beta_1}|\partial_{\bar{\xi}}w|^{\frac{2}{3}}|\partial^2_{\bar{\eta}\bar{\eta}}  w |^2\zeta^{10}-  C_{\beta_1}\frac{|\partial^2_{\bar{\xi}\bar{\eta}}  w |^2}{|\partial_{\bar{\xi}}w|^{\frac{2}{3}}+1}\zeta^{10}
\\&+\gamma_1 |\partial_{\bar{\xi}}w|^2\zeta^{10}-C\zeta^{9}|\partial_{\bar{\xi}}w|^{\frac{4}{3}}
-\zeta^{10}C
\\ \geq &- C_{\beta_1}|\partial_{\bar{\xi}}w|^{\frac{2}{3}}|\partial^2_{\bar{\eta}\bar{\eta}}  w |^2\zeta^{10}- C_{\beta_1}\frac{|\partial^2_{\bar{\xi}\bar{\eta}}  w |^2}{|\partial_{\bar{\xi}}w|^{\frac{2}{3}}+1}\zeta^{10}\\  &+\gamma_1 |\partial_{\bar{\xi}}w|^2\zeta^{10}-C.\end{split}
\end{align}

Next we consider the second term in $f.$ By direct calculations, we have
\begin{align*}
L ((|\partial_{\bar{\xi}}  w|^2+1)^{\frac{2}{3}}\zeta^{10})\geq& \gamma_2 \frac{|\partial^2_{\bar{\xi}\bar{\eta}}  w |^2}{|\partial_{\bar{\xi}}w|^{\frac{2}{3}}+1}\zeta^{10}-\beta_2 |\partial_{\bar{\xi}}w|^2\zeta^{10}-C_{\beta_2}|\partial_{\bar{\xi}}w|^{\frac{2}{3}}|\partial^2_{\bar{\eta}\bar{\eta}}  w |^2\zeta^{10}\\&-C_{T^*}\zeta^{9}|\partial_{\bar{\xi}}w|^{\frac{4}{3}}-C
\\ \geq& \gamma_2 \frac{|\partial^2_{\bar{\xi}\bar{\eta}}  w |^2}{|\partial_{\bar{\xi}}w|^{\frac{2}{3}}+1}\zeta^{10}-\beta_2 |\partial_{\bar{\xi}}w|^2\zeta^{10}-C_{\beta_2}|\partial_{\bar{\xi}}w|^{\frac{2}{3}}|\partial^2_{\bar{\eta}\bar{\eta}}  w |^2\zeta^{10}-C.
\end{align*}
Now we consider the third term in $f$. By \eqref{bdetaw}, we have
\begin{align*}
\big[L (|\partial_{\bar{\xi}}  w|^2+B)^{\frac{1}{3}}\big]|\partial_{\bar{\eta}}  w|^2
\geq& -\beta_3|\partial_{\bar{\xi}}  w|^{\frac{2}{3}}|\partial_{\bar{\eta}\bar{\eta}}^2w|^2 -C_{\beta_3}|\partial_{\bar{\xi}}  w|^{\frac{2}{3}}-C\frac{|\partial_{\bar{\xi}\bar{\eta}}^2w| ^2}{(|\partial_{\bar{\xi}}  w|^2+B  )^{\frac{2}{3}}}\\ \geq& -\beta_3|\partial_{\bar{\xi}}  w|^{\frac{2}{3}}|\partial_{\bar{\eta}\bar{\eta}}^2w|^2 -\beta_4|\partial_{\bar{\xi}}  w|^{2}-C\frac{|\partial_{\bar{\xi}\bar{\eta}}^2w| ^2}{(|\partial_{\bar{\xi}}  w|^2+B  )^{\frac{2}{3}}}-C.
\end{align*}
Thanks to
\begin{align*}
2w^2\big(\partial_{\bar{\eta}} (|\partial_{\bar{\xi}}  w|^2+B)^{\frac{1}{3}}\big)2\partial_{\bar{\eta}}  w\partial^2_{\bar{\eta}\bar{\eta}}  w&=2w^2\frac{4\partial_{\bar{\xi}}w\partial_{\bar{\xi}\bar{\eta}}^2 w}{3(|\partial_{\bar{\xi}}  w|^2+B)^{\frac{2}{3}}}\partial_{\bar{\eta}}  w\partial^2_{\bar{\eta}\bar{\eta}}  w,
\end{align*}
we have
\begin{align*}|
2w^2\big(\partial_{\bar{\eta}} (|\partial_{\bar{\xi}}  w|^2+B)^{\frac{1}{3}}\big)2\partial_{\bar{\eta}}  w\partial^2_{\bar{\eta}\bar{\eta}}  w|
&\leq\frac{C|\partial_{\bar{\xi}\bar{\eta}}^2 w|}{(|\partial_{\bar{\xi}}  w|^2+B)^{\frac{1}{6}}}|\partial^2_{\bar{\eta}\bar{\eta}}  w|\leq\frac{\beta_5|\partial_{\bar{\xi}\bar{\eta}}^2 w|^2}{(|\partial_{\bar{\xi}}  w|^2+B)^{\frac{1}{3}}}+C_{\beta_5}|\partial^2_{\bar{\eta}\bar{\eta}}  w|^2.\end{align*}
On the other hand,
\begin{align*}
\big[L|\partial_{\bar{\eta}}  w|^2\big](|\partial_{\bar{\xi}}  w|^2+B)^{\frac{1}{3}}
&\geq \big[- C|\partial_{\bar{\xi}}w|-C|\partial^2_{\bar{\eta}\bar{\eta}}  w | +2w^2|\partial_{\bar{\eta}\bar{\eta}}^2w|^2\big](|\partial_{\bar{\xi}}  w|^2+B)^{\frac{1}{3}}\\&\geq\big[- C|\partial_{\bar{\xi}}w|-C +w^2|\partial_{\bar{\eta}\bar{\eta}}^2w|^2\big](|\partial_{\bar{\xi}}  w|^2+B)^{\frac{1}{3}} \\&\geq \gamma_3|\partial_{\bar{\eta}\bar{\eta}}^2w|^2|\partial_{\bar{\xi}}  w|^{\frac{2}{3}}-C|\partial_{\bar{\xi}}  w|^{\frac{5}{3}}-C\\&\geq \gamma_3|\partial_{\bar{\eta}\bar{\eta}}^2w|^2|\partial_{\bar{\xi}}  w|^{\frac{2}{3}}- \beta_6|\partial_{\bar{\xi}}w|^2-C.
\end{align*}
Summing up, we obtain
\begin{align*}
&L\big[(|\partial_{\bar{\xi}}  w|^2+B)^{\frac{1}{3}}|\partial_{\bar{\eta}}  w|^2\big]
\\ &\geq-\beta_3|\partial_{\bar{\xi}}  w|^{\frac{2}{3}}|\partial_{\bar{\eta}\bar{\eta}}^2w|^2 -\beta_4|\partial_{\bar{\xi}}  w|^{2}-C_1\frac{|\partial_{\bar{\xi}\bar{\eta}}^2w| ^2}{(|\partial_{\bar{\xi}}  w|^2+B  )^{\frac{2}{3}}}-\frac{\beta_5|\partial_{\bar{\xi}\bar{\eta}}^2 w|^2}{(|\partial_{\bar{\xi}}  w|^2+B)^{\frac{1}{3}}}\\&\quad-C_{\beta_5}|\partial^2_{\bar{\eta}\bar{\eta}}  w|^2+\gamma_3|\partial_{\bar{\eta}\bar{\eta}}^2w|^2|\partial_{\bar{\xi}}  w|^{\frac{2}{3}}-C- \beta_6|\partial_{\bar{\xi}}w|^2\\ &\geq \gamma_4|\partial_{\bar{\eta}\bar{\eta}}^2w|^2|\partial_{\bar{\xi}}  w|^{\frac{2}{3}}-(\beta_4+ \beta_6)|\partial_{\bar{\xi}}  w|^{2}-\frac{(\beta_5+\beta_7)|\partial_{\bar{\xi}\bar{\eta}}^2 w|^2}{(|\partial_{\bar{\xi}}  w|^2+B)^{\frac{1}{3}}}-C_{\beta_5}|\partial^2_{\bar{\eta}\bar{\eta}}  w|^2-C
\end{align*}
by taking $B$ large and $\beta_3$ small so that
\begin{equation}\label{o1}
  \frac{C_1}{(|\partial_{\bar{\xi}}  w|^2+B)^{\frac{1}{3}}}\leq \beta_7,\quad\,\beta_3\ll \gamma_3.
\end{equation}
Hence,\begin{align*}
L [(|\partial_{\bar{\xi}}  w|^2+B)^{\frac{1}{3}}|\partial_{\bar{\eta}}  w|^2\zeta^{10}] \geq&\gamma_4|\partial_{\bar{\eta}\bar{\eta}}^2w|^2|\partial_{\bar{\xi}}  w|^{\frac{2}{3}}\zeta^{10}-C -(\beta_4+ \beta_6)|\partial_{\bar{\xi}}  w|^{2}\zeta^{10}\\&-\frac{(\beta_5+\beta_7)|\partial_{\bar{\xi}\bar{\eta}}^2 w|^2}{(|\partial_{\bar{\xi}}  w|^2+B)^{\frac{1}{3}}}\zeta^{10}-C_{\beta_5}|\partial^2_{\bar{\eta}\bar{\eta}}  w|^2\zeta^{10}-C\zeta^{9}(|\partial_{\bar{\xi}}  w|^2+B)^{\frac{1}{3}}
\\\geq &\gamma_4|\partial_{\bar{\eta}\bar{\eta}}^2w|^2|\partial_{\bar{\xi}}  w|^{\frac{2}{3}}\zeta^{10} -(\beta_4+ \beta_6+\beta_{10})|\partial_{\bar{\xi}}  w|^{2}\zeta^{10}\\&-\frac{(\beta_5+\beta_7)|\partial_{\bar{\xi}\bar{\eta}}^2 w|^2}{(|\partial_{\bar{\xi}}  w|^2+B)^{\frac{1}{3}}}\zeta^{10}-C_{\beta_5}|\partial^2_{\bar{\eta}\bar{\eta}}  w|^2\zeta^{10}-C.
\end{align*}

For the fourth term in $f$, we have
\begin{align*}
L(|\partial_{\bar{\eta}}  w|^2\zeta^{10})
&\geq (- C|\partial_{\bar{\xi}}w|-C|\partial^2_{\bar{\eta}\bar{\eta}}  w | +2w^2|\partial_{\bar{\eta}\bar{\eta}}^2w|^2)\zeta^{10}-\zeta^{9}C\\
&\geq - \beta_9|\partial_{\bar{\xi}}w|^2\zeta^{10}-C+w^2|\partial_{\bar{\eta}\bar{\eta}}^2w|^2\zeta^{10}.
\end{align*}
For the last two terms in $f$, we have
\begin{align*}
   L( -M_4\bar{t}+\bar{\eta})=M_4.
\end{align*}

Finally, let us fix the coefficients in the following order.\smallskip
\begin{itemize}

\item[(1)] Fix $\beta_1$ small depending on $a,A$. See \eqref{beta1}.

\item[(2)] Fix $M_1$ large such that $\frac{M_1}{100}>\frac{C_{\beta_1}}{\gamma_2}.$

\item[(3)] Fix $\beta_2$ such that $\beta_2<\frac{\gamma_1}{100M_1}.$

\item[(4)]  Fix $M_2$ large such that $M_2>\frac{M_1C_{\beta_2}+C_{\beta_1}}{\gamma_4}.$

 \item[(5)]  Fix $\beta_5,\beta_7$ such that $\beta_5+\beta_7<\frac{M_1}{2M_2}\gamma_2$ and then fix $B$ by \eqref{o1} accordingly.

\item[(6)] Fix $\beta_4,\beta_6,\beta_{10}$ such that $\beta_4+\beta_6+\beta_{10}<\frac{\gamma_1}{100M_2}.$

\item[(7)] Fix $M_3$ large such that $M_3>\frac{C_{\beta_5}}{a^2}M_2.$

\item[(8)] Fix $\beta_9$ such that $\beta_9<\frac{\gamma_1}{100M_3}.$

\item[(9)] Fix $M_4$ large such that $M_4>C$.

\end{itemize}

 \smallskip

 With such choices of the constants, we can deduce (C). Combining (A), (B) and (C), we infer that $f\leq C$. Hence, $$M_1(|\partial_{\bar{\xi}}  w|^2+1)^{\frac{2}{3}} \zeta^{10}\leq C.$$
\\ Restricting to the point $\bar{x}_0=(\eta_0,0,0)$ and transforming back to the coordinate $(\eta,\xi,t)$, we have $|\partial_{\xi}w(x_0)|\leq C.$ Since $C$ is independent of the choice of $x_0,$ we have
$ |\partial_{\xi}w|\leq C.$
\end{proof}

%

\section{Proof of Theorem \ref{thm:main1} and Theorem \ref{thm:main2}}

Theorem \ref{thm:main2} is a direct consequence of Proposition \ref{prop:mono} and Proposition \ref{prop:grad}. Let us  prove Theorem \ref{thm:main1}.\smallskip

The local well-posedness part has been essentially  proved in \cite{CWZ}.  So, it suffices to prove the blow-up criterion.
We introduce
\beno
&&\cE(t)=\|w\|_{H^{3,0}_{\nu}}^2+\|\tu\|_{H^{2,0}_\om}^2+\|\tu\|_{H_\mu^{1,0}}^2+\|\pa_y\tu\|_{H_\mu^{1,0}}^2+\|\tu_t\|_{H_\mu^{1,0}}^2,\\
&&\cD(t)=\|\pa_yw\|_{H^{3,0}_{\nu}}^2+\|\pa_y\tu\|^2_{H^{2,0}_\om}+\|\pa_y\tu\|_{H_\mu^{1,0}}^2+\|\tu_t\|_{H_\mu^{1,0}}^2+\|\pa_y\tu_t\|_{H_\mu^{1,0}}^2.
\eeno
Then we deduce from Proposition \ref{prop:energy-w-s}, Proposition \ref{prop:u-energy-s}  and Lemma \ref{lem:relation} that
\begin{align}\label{eq:energy-S1}
\f d {dt}\cE(t)+\cD(t)\le& \cP(A(t))\cE(t)+\cP(A(t))\cE(t)^\f12\|\tu\|_{H^{2,2}_\mu}.
\end{align}
Here $\cP(\cdot)$ is some increasing function. By  (\ref{eq:interpolation}), we have
\beno
\|\tu\|_{H^{2,2}_\mu}\le \|\tu\|_{H^{3,1}_\mu}^\f12\|\tu\|_{H^{1,3}_\mu}^\f12,
\eeno
which along with Lemma \ref{lem:relation} implies that
\beno
\|\tu\|_{H^{2,2}_\mu}\le \cP(A(t))\cE(t)^\f12+\|\pa_y\tu_t\|_{H^{1,0}_\mu}
.\eeno
Then we infer from (\ref{eq:energy-S1}) and  that
\begin{align*}
&\f d {dt}\cE(t)+\cD(t)\le \cP(A(t))\cE(t),
\end{align*}
from which and Gronwall's inequality, we deduce that
\beno
\cE(t)+\int_0^t\cD(s)ds\le C\cE(0)
\eeno
for any $t\in [0,T^*)$ if $\sup_{t\in [0,T^*)}A(t)\le C$. This implies that the solution can be  extended after $t=T^*$. The proof of Theorem \ref{thm:main1} is completed.
\ef\smallskip

\section*{Acknowledgments}

The authors thank Professors Zhouping Xin and Liqun Zhang for their helpful comments. Z. Zhang is partially supported by NSF of China under Grant 11425103.

 \end{document}